\documentclass[final]{siamltex}
\usepackage{stmaryrd,amsmath,amssymb,ifthen,calc,graphicx,epsfig,ifpdf}
\def\MAF{{\rm MAF}}

\def\norm#1{\|#1\|}
\let\phi=\varphi
\let\ol=\overline

\newtheorem{remark}[theorem]{Remark}

\newtheorem{assumption}[theorem]{Assumption}
\newcommand{\R}{\mathbb{R}}
\newcommand{\Rn}{\mathbb{R}^n}

\newcommand{\Rp}{\mathbb{R}_+}
\newcommand{\Rnp}{\ensuremath{\mathbb{R}^n_+}}

\newcommand{\N}{\mathbb{N}}

\newcommand{\K}{\ensuremath{\mathcal{K}}}
\newcommand{\KL}{\ensuremath{\mathcal{KL}}}
\newcommand{\Kinf}{\ensuremath{\mathcal{K}_\infty}}
\newcommand{\Knn}[1][n]{\ensuremath{(\K\cup \{0\})^{#1\times #1}}}
\newcommand{\Kinfnn}[1][n]{\ensuremath{(\Kinf\cup \{0\})^{#1\times #1}}}

\newcommand{\id}{\ensuremath{\mbox{id}}}
\newcommand{\conv}{\ensuremath{\mbox{conv}}}

\newcommand{\einsnorm}[2]{\ensuremath{
  \!\!\;\!\!\!\;
  \left\bracevert\!\!\!\!\!\left\bracevert
      \!
      #1\ifthenelse{\equal{x#2}{x}}{}{(#2)}
      \!
    \right\bracevert\!\!\!\!\!\right\bracevert
  \!\!\;\!\!\!\;
}}

\let\epsilon=\varepsilon

\newcounter{enumctr}

\renewcommand{\theenumi}{(\roman{enumi})}

\renewcommand{\theenumii}{(\alph{enumii})}

\title{
Small gain theorems for large scale systems and construction of
ISS
Lyapunov functions%
\thanks{%
 This work was done while B.~S.~R\"uffer was with the
 School of Electrical Engineering and
 Computer Science, University of Newcastle, Callaghan, NSW 2308,
 Australia. %
 Sergey Dashkovskiy has been
 supported by the German Research Foundation (DFG) as part of the
 Collaborative Research Center 637 "Autonomous Cooperating Logistic
 Processes: A Paradigm Shift and its Limitations" (SFB 637).
 B.~S.~R\"uffer has been supported by the Australian Research
 Council under grant DP0771131. Fabian Wirth is supported by 
 the DFG priority programme 1305 ``Control Theory of Digitally Networked Dynamical Systems''.}}

\author{Sergey N. Dashkovskiy\thanks{ Universit{\"a}t Bremen, Zentrum
  f{\"u}r Technomathematik, Postfach 330440, 28334 Bremen, Germany,
  {\tt dsn@math.uni-bremen.de}}  \and %
Bj{\"o}rn S. R{\"u}ffer\thanks{%
Department of Electrical and Electronic Engineering, %
University of Melbourne, %
Parkville, Victoria 3010, %
Australia, %
{\tt brueffer@unimelb.edu.au}}%
\and%
Fabian R. Wirth\thanks{Institut f\"ur Mathematik, Universit\"at
  W{\"u}rzburg, Am Hubland, D-97074 W\"urzburg, Germany, {\tt
    wirth@mathematik.uni-wuerzburg.de}} }

\begin{document}

\maketitle

\begin{abstract}
We consider interconnections of $n$ nonlinear
subsystems in the input-to-state stability (ISS) framework. For each
subsystem an ISS Lyapunov function is given that treats the other
subsystems as independent inputs. 
A gain matrix is used to encode the
mutual dependencies of the systems in the network. Under a small
gain assumption on the monotone operator induced by the gain matrix,
a locally Lipschitz continuous ISS Lyapunov function is obtained constructively
for the entire network by appropriately scaling the individual
Lyapunov functions for the subsystems.
The results are obtained in a general formulation of ISS, the cases of summation, maximization
and separation with respect to external gains are obtained as corollaries.
\end{abstract}



\pagestyle{myheadings} \thispagestyle{plain} \markboth{S.N.
DASHKOVSKIY, B.S. R{\"U}FFER, F.R. WIRTH}{ISS LYAPUNOV FUNCTIONS
FOR INTERCONNECTED SYSTEMS}

\section{Introduction}
\label{sec:introduction}
In many applications large scale systems are obtained through the
interconnection of a number of smaller components.  The stability
analysis of such interconnected systems may be a difficult task
especially in the case of a large number of subsystems, arbitrary
interconnection topologies, and nonlinear subsystems.

One of the earliest tools in the stability analysis of feedback
interconnections of nonlinear systems are small gain theorems.
Such results have been obtained by many authors starting with
\cite{Zam66}. These results are classically built on the notion of
$L^p$ gains, see \cite{CCF07} for a recent, very readable account
of the developments in this area. While most small gain results
for interconnected systems yield only sufficient conditions, in
\cite{CCF07} it has been shown in a behavioral framework how the
notion of gains can be modified so that the small gain condition
is also necessary for robust stability.

Small gain theorems for large scale systems have been developed,
e.g., in \cite{Sil91,Vid81,RHL77}. In \cite{Sil91} the notions of
connective stability and stabilization are introduced for
interconnections of linear systems using the concept of vector
Lyapunov functions. In \cite{RHL77} stability conditions in terms
of Lyapunov functions of subsystems have been derived. For the
linear case characterizations of quadratic stability of large
scale interconnections have been obtained in \cite{HinrPrit09}. A
common feature of these references is that the gains describing
the interconnection are essentially linear. With the introduction
of the concept of input-to-state stability in \cite{Son89}, it has
become a common approach to consider gains as a nonlinear
functions of the norm of the input. In this nonlinear case small
gain results have been derived first for the interconnection of
two systems in \cite{JTP94,Tee96}. A Lyapunov version of the same
result is given in \cite{JMW96}. A general small gain condition
for large scale ISS systems has been presented in \cite{DRW-mcss}.
Recently, such arguments have been used in the stability analysis
of observers \cite{APA07}, in the stability analysis of
decentralized model predictive control \cite{RMS07} and in the
stability analysis of groups of autonomous vehicles.

During the revision of this paper it came to our attention that,
following the first general small gain theorems for networks
\cite{Pot03,Tee05,DRW05-cdcecc,DRW06b,DRW07-NOLCOS,DRW-mcss}, other
generalizations of small gain results based on similar ideas have
been obtained very recently using the maximization
formulation of ISS: A generalized small gain theorem for
output-Lagrange-input-to-output stable systems in network
interconnections has been obtained in
\cite{JiangWang:2008:A-generalization-of-the-nonlinear-small-:}. 
In this reference the authors study ISS in the maximization framework and conclude 
ISS from a small gain condition in the cycle formulation. It has been noted
in \cite{DRW06b} that in the maximum case 
the cycle condition is equivalent to the operator 
condition examined here.
An extension of generalized small gain results to retarded functional
differential equations based on the more general cycle condition and
vector Lyapunov functions has recently been obtained in
\cite{KarafyllisJiang:5:A-Vector-Small-Gain-Theorem-for-General-:}. 
In this reference a construction of a Lyapunov function is shown which
takes a different approach to the construction of an overall
Lyapunov function. This construction depends vitally on the use of 
the maximum formulation of ISS.

In this paper we present sufficient conditions for the existence
of an ISS Lyapunov function for a system obtained as the
interconnection of many subsystems. The results are of interest in
two ways. First, it is shown that a small gain condition is
sufficient for input-to-state stability of the large scale system
in the Lyapunov formulation. Secondly, an explicit formula for an
overall Lyapunov function is given.  As the dimensions of the
subsystems are essentially lower than the dimension of their
interconnection, finding Lyapunov functions for them may be an
easier task than for the whole system.

Our approach is based on the notion of {\em input-to-state
stability} (ISS) introduced in \cite{Son89} for nonlinear systems
with inputs.  A system is ISS if, roughly speaking, it is globally
asymptotically stable in the absence of inputs (so-called 0-GAS)
and if any trajectory eventually enters a ball centered at the
equilibrium point and with radius given by a monotone continuous
function, the gain, of the size of the input (the so-called
\emph{asymptotic gain property}), cf.\ \cite{SoW96}.

The concept of ISS turned out to be particularly well suited to
the investigation of interconnections. For example, it is known
that cascades of ISS systems are again ISS \cite{Son89} and small
gain results have been obtained. We briefly review the results of
\cite{JTP94,JMW96} in order to explain the motivation for the
approach of this paper. Both papers study a feedback
interconnection of two ISS systems as represented in
Figure~\ref{feedbackfig}.

\begin{figure}[h]
\centering
      {\includegraphics[width=.5\columnwidth]{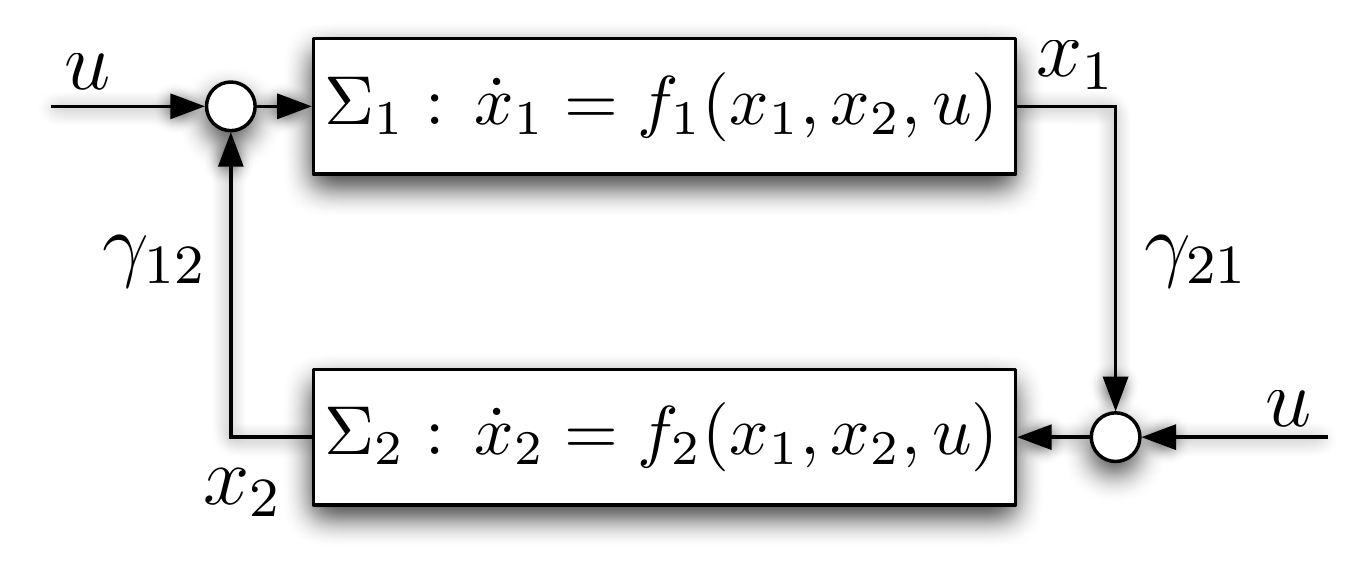}}
\caption{Feedback interconnection of two ISS systems with gains
$\gamma_{12}$ from $\Sigma_2$ to $\Sigma_1$ and $\gamma_{21}$ from
$\Sigma_1$ to $\Sigma_2$.} \label{feedbackfig}
\end{figure}

The small gain condition in \cite{JTP94} is that the composition
of the gain functions $\gamma_{12},\gamma_{21}$ is less than
identity in a robust sense.
We denote the composition of functions $f,g$ by $\circ$, that is, $(f\circ g)(x):=f(g(x))$.
The small gain condition then is that if on $(0,\infty)$ we have
\begin{equation}
(\id+\alpha_1) \circ \gamma_{12} \circ (\id+\alpha_2) \circ
\gamma_{21}<\id\,, \label{eq:1}
\end{equation}
for suitable $\Kinf$ functions $\alpha_1,\alpha_2$  then
the feedback system is ISS with respect to the external inputs.

In this paper we concentrate on the  equivalent definition of ISS
in terms of ISS Lyapunov functions \cite{SoW96}. The small gain
theorem for ISS Lyapunov functions from \cite{JMW96} states that
if on $(0,\infty)$ the small gain condition
\begin{equation}
\gamma_{12}\circ\gamma_{21}<\id\label{eq:2}
\end{equation}
is satisfied then an ISS Lyapunov function may be explicitly
constructed as follows. Condition \eqref{eq:2} is equivalent to
$\gamma_{12} < \gamma_{21}^{-1}$ on $(0,\infty)$. This permits to
construct a function $\sigma_2\in\Kinf$ such that $\gamma_{21} <
\sigma_2 < \gamma_{12}^{-1}$ on $(0,\infty)$, see
Figure~\ref{figtwosigmas}. An ISS Lyapunov function is then
defined by scaling and taking the maximum, that is, by setting
$V(x) = \max\{V_1(x_1), \sigma_2^{-1}(V_2(x_2)) \}$.
This ISS Lyapunov function describes stability properties of the
whole interconnection. In particular, given an input $u$, 
it can be seen
how fast the corresponding trajectories converge to the
neighborhood and how large this neighborhood is.
\begin{figure}[h]
\centering
\includegraphics[width=.4\columnwidth]{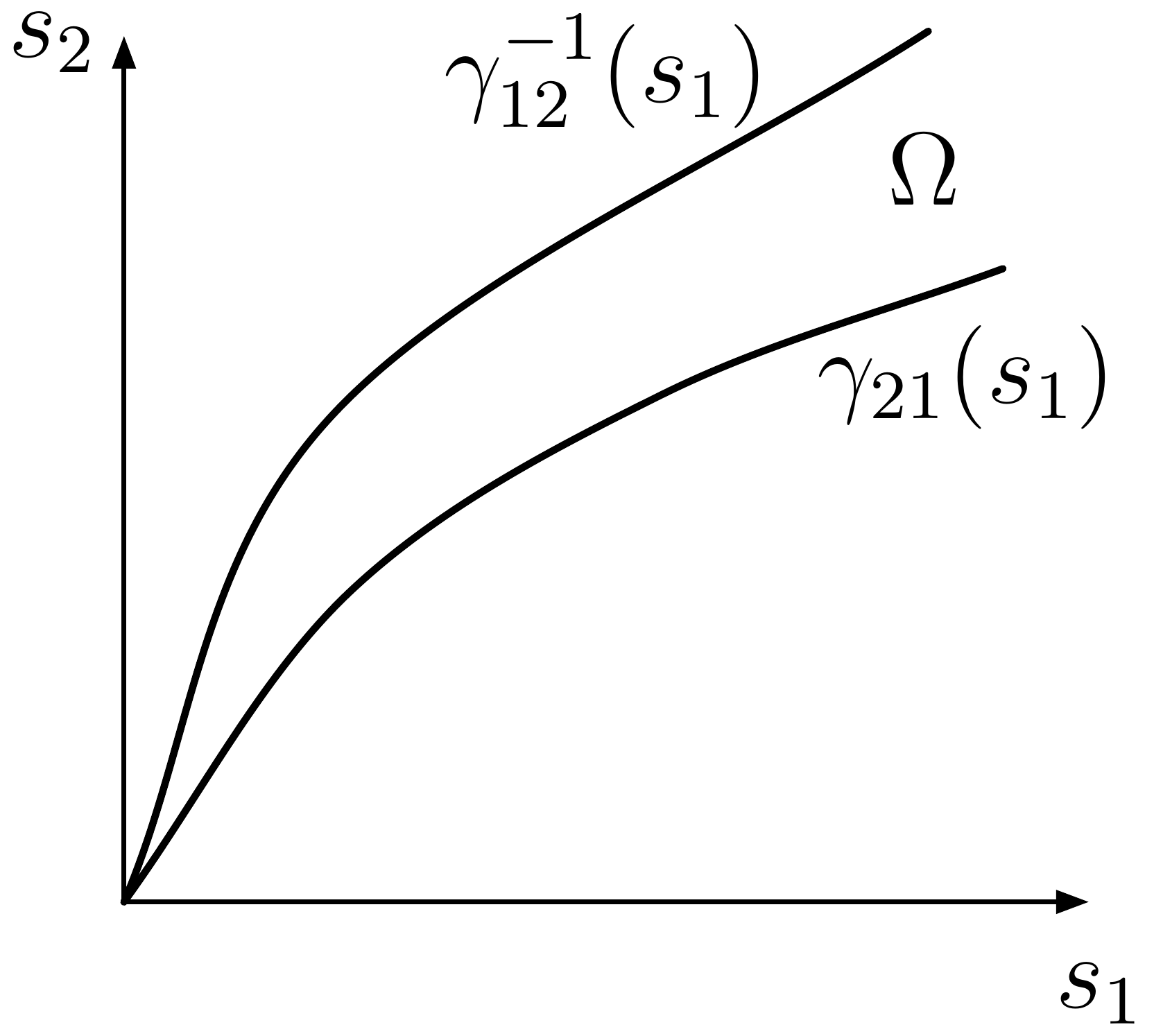}
\caption{Two gain functions satisfying \eqref{eq:2}.}
\label{figtwosigmas}
\end{figure}

At first sight the difference between the small gain conditions in
\eqref{eq:1} from \cite{JTP94} and \eqref{eq:2} from \cite{JMW96}
appears surprising. This might lead to the impression that the
difference comes from studying the problem in a trajectory based
or Lyapunov based framework. This, however, is not the case; the
reason for the difference in the conditions is a result of the
formulation of the ISS condition. In \cite{JTP94} a summation
formulation was used for the trajectory based case. In the maximization formulation 
of the trajectory case the small gain condition is again
\eqref{eq:2}, \cite{DRW-mcss}. In \cite{JMW96} the Lyapunov formulation
is investigated using maximization, the corresponding result for summation
is Corollary~\ref{sumcor} below requiring condition \eqref{eq:1}.

In order to generalize the existing results it is useful to
reinterpret the approach of \cite{JMW96}: note that the gains may
be used to define a matrix
\[ \Gamma := \begin{pmatrix}
0 & \gamma_{12}\\
\gamma_{21} & 0
\end{pmatrix}\,,\]
which defines in a natural way a monotone operator on $\R_+^2$. In
this way an alternative characterization of the area between
$\gamma_{21}$ and $\gamma_{12}^{-1}$ in Figure~\ref{figtwosigmas}
is that it is the area where $\Gamma(s) < s$ (with respect to the
natural ordering in $\R_+^2$). Thus the problem of finding
$\sigma_2$ may be interpreted as the problem of finding a path
$\sigma : r \mapsto (r,\sigma_2(r)), r \in (0, \infty)$ such that
$\Gamma \circ \sigma < \sigma$.

We generalize this constructive procedure for a Lyapunov function
in several directions. First the number of subsystems entering the
interconnection will be arbitrary. Secondly, the way in which the
gains of subsystem $i$ affect subsystem $j$ will be formulated in
a general manner using the concept of {\em monotone aggregation
functions}. This class of functions allows for a unified treatment
of summation, maximization or other ways of formulating ISS
conditions. Following the matrix interpretation this leads to a
monotone operator $\Gamma_\mu$ on $\Rnp$. The crucial thing to
find is a sufficiently regular path $\sigma$ such that $\Gamma_\mu
\circ \sigma < \sigma$. This allows for a scaling of the Lyapunov
functions for the individual subsystems to obtain one for the
large scale system.

Small gain conditions on $\Gamma_\mu$ as in
\cite{DRW05-cdcecc,DRW-mcss} yield sufficient conditions that
guarantee that the construction of $\sigma$ can be performed. 
\label{revision:commenting-on-differences-to-MCSS-and-CDCECC}
However, in \cite{DRW05-cdcecc,DRW-mcss} the trajectory
formulation of ISS has been studied, and the main technical
ingredient was, essentially, to prove bounds on $(\id
-\Gamma_{\mu})^{-1}$. The sufficient condition for the existence
of the path $\sigma$ turns out to be the same, but the path itself
had not been used in \cite{DRW05-cdcecc,DRW-mcss}. In fact, the
line of argument used there is completely different. 
It is shown in \cite{07-R-DISS} that the results of
\cite{DRW-mcss} also hold for the more general ISS formulation
using monotone aggregation functions. The condition requires
essentially that the operator is not greater or equal to the identity
in a robust sense.  The construction of $\sigma$ then relies on a
rather delicate topological argument. What is obvious for the
interconnection of two systems is not that clear in higher
dimensions. It can be seen that the small gain condition imposed
on the interconnection is actually a sufficient condition that
allows for the application of the Knaster-Kuratowski-Mazurkiewicz
theorem, see \cite{DRW-mcss,07-R-DISS} for further details. We
show in Section~\ref{sec:remarks-case-three} how the construction
works for three subsystems, but it is fairly clear that this
methodology is not something one would like to carry out in higher
dimensions. 
In the maximization formulation a viable alternative is the
approach pursued by \cite{KarafyllisJiang:5:A-Vector-Small-Gain-Theorem-for-General-:}.

The construction of the Lyapunov function is explicit once the
scaling function $\sigma$ is known. Thus to have a really
constructive procedure a way of constructing $\sigma$ is required.
We do not study this problem here, but note that based on an
algorithm by Eaves \cite{Eav72} it actually possible to turn this
mere existence result into a (numerically) constructive method
\cite{07-R-DISS,DRW-CDC07}. Using the algorithm by Eaves and the
technique of
Proposition~\ref{prop:psi-pathwise-connected-AND-finite-length-Omega-path},
it is then possible to construct such a vector function (but of
finite-length) numerically, see \cite[Chapter 4]{07-R-DISS}. This
will be treated in more detail in future work.

The paper is organized as follows. The next section introduces the
necessary notation and basic definitions, in particular the notion
of monotone aggregation functions (MAFs) and different
formulations of ISS. Section~\ref{examples} gives some motivating
examples that also illustrate the definitions of the
Section~\ref{sec:preliminaries} and explain how different MAFs
occur naturally for different problems. In
Section~\ref{sec:monot-oper-gener} we introduce small gain
conditions given in terms of monotone operators that naturally
appear in the definition of ISS.
Section~\ref{sec:lyapunov-functions} contains the main results,
namely the existence of the vector scaling function $\sigma$ and
the construction of an ISS Lyapunov function. In this section we
concentrate on strongly connected networks which are easier to
deal with from a technical point of view. Once this case has been
resolved it is shown in Section~\ref{sec:reducible} how simply
connected networks may be treated by studying the strongly
connected components.

The actual construction of $\sigma$ is given in
Section~\ref{sec:pathconstruction} to postpone the topological
considerations until after applications to interconnected ISS
systems have been considered in
Section~\ref{sec:appl-gener-small}. Since the topological
difficulties can be avoided in the case $n=3$ we treat this case
briefly in Section~\ref{sec:remarks-case-three} to show a simple
construction for $\sigma$. Section~\ref{sec:conclusions} concludes
the paper.

\section{Preliminaries}
\label{sec:preliminaries}
\subsection{Notation and conventions}
Let $\R$ be the field of real numbers and $\R^n$ the vector space
of real column vectors of length $n$. We denote the set of
nonnegative real numbers by $\R_+$ and $\Rnp:= (\R_+)^n$ denotes
the positive orthant in $\R^n$.
On $\Rnp$ the standard partial order is defined as follows. For vectors $v,w\in\R^n$ we
denote
\begin{gather*}
v\geq w:\!\iff  v_i\geq w_i\mbox{ for }i=1,\ldots,n,\\
v> w :\!\iff v_i> w_i \mbox{ for }i=1,\ldots,n,\\
v\gneqq w :\!\iff v\geq w\mbox{ and }v\ne w.
\end{gather*}
The maximum of two vectors or matrices is to be understood 
component-wise.
By $|\cdot|$ we denote the 1-norm on $\R^n$ and by $S_r$ the
induced sphere of radius $r$ in $\R^n$ intersected with \Rnp,
which is an $(n-1)$-simplex. On $\Rnp$ we denote by
$\pi_I:\Rnp\to\R^{\#I}_+$ the
\emph{projection}\index{projection}\index{1p@$\pi_I$} of the
coordinates in $\Rnp$ corresponding to the indices in
$I\subset\{1,\ldots,n\}$ onto $\R^{\#I}$.

The standard scalar product in $\Rn$ is denoted by $\langle
\cdot,\cdot \rangle$.  By $U_\varepsilon(x)$ we denote the open
ball of radius $\varepsilon$ around $x$ with respect to the
Euclidean norm $\|\cdot\|$. The induced operator norm, i.e. the
spectral norm, of matrices is also denoted by $\|\cdot\|$.

The space of measurable and essentially bounded functions is
denoted by $L^\infty$ with norm $\|\cdot\|_\infty$\,. To state the
stability definitions that we are interested in three
sets of comparison functions are used: $\K = \{ \gamma : \Rp\to\Rp, \gamma$
is continuous, strictly increasing, and $\gamma(0)=0 \}$ and
$\Kinf=\{ \gamma\in\K: \gamma\mbox{ is unbounded}\}$.  A function
$\beta:\Rp\times\Rp\to\Rp$ is of class $\mathcal{KL}$, if it is of
class $\K$ in the first argument and strictly decreasing to zero
in the second argument.  We will call a function $V: \R^N \to
\R_+$ {\em proper and positive definite} if there are
$\psi_1,\psi_2\in \Kinf$ such that
\begin{equation*}
\psi_1(\| x\|) \leq V(x) \leq \psi_2(\| x \|) \,, \quad \forall x
\in \R^N\,.
\end{equation*}
A function $\alpha: \R_+ \to \R_+$ is called {\em positive
definite} if it is continuous and satisfies $\alpha(r) = 0$ if and
only if $r=0$.

\subsection{Problem Statement}

We consider a finite set of interconnected systems with state
$x=\begin{pmatrix} x_1^T, \ldots, x_n^T \end{pmatrix}^T$, where
$x_i \in \R^{N_i}$, $i=1,\ldots,n$ and $N:=\sum N_i$. For
$i=1,\ldots,n$ the dynamics of the $i$-th subsystem is given by
\begin{equation}
\label{eq:3}
\Sigma_i:~\dot x_i = f_i( x_1,\ldots,x_n,u),
\quad x\in\R^{N},\;u\in\R^{M},\; f_i:\R^{N+M}\to\R^{N_i} \,.
\end{equation}

For each $i$ we assume unique existence of solutions and forward
completeness of $\Sigma_i$ in the following sense.  If we
interpret the variables $x_j$, $j\ne i$, and $u$ as unrestricted
inputs, then this system is assumed to have a unique solution
defined on $[0,\infty)$ for any given initial condition $x_i(0)\in
\R^{N_i}$ and any $L^\infty$-inputs $x_j:[0,\infty)\to\R^{N_j},
j\ne i$, and $u:[0,\infty)\to\R^M$. This can be guaranteed for
instance by suitable Lipschitz and growth conditions on the $f_i$.
It will be no restriction to assume that all systems have the same
(augmented) external input $u$.

We write the interconnection of subsystems \eqref{eq:3} as
\begin{equation}
\label{eq:4}
\Sigma:~\dot x = f(x,u), \quad f:\R^{N+M}\to\R^{N}.
\end{equation}
\begin{figure}[h]
\centering
\includegraphics[width=.8\columnwidth]{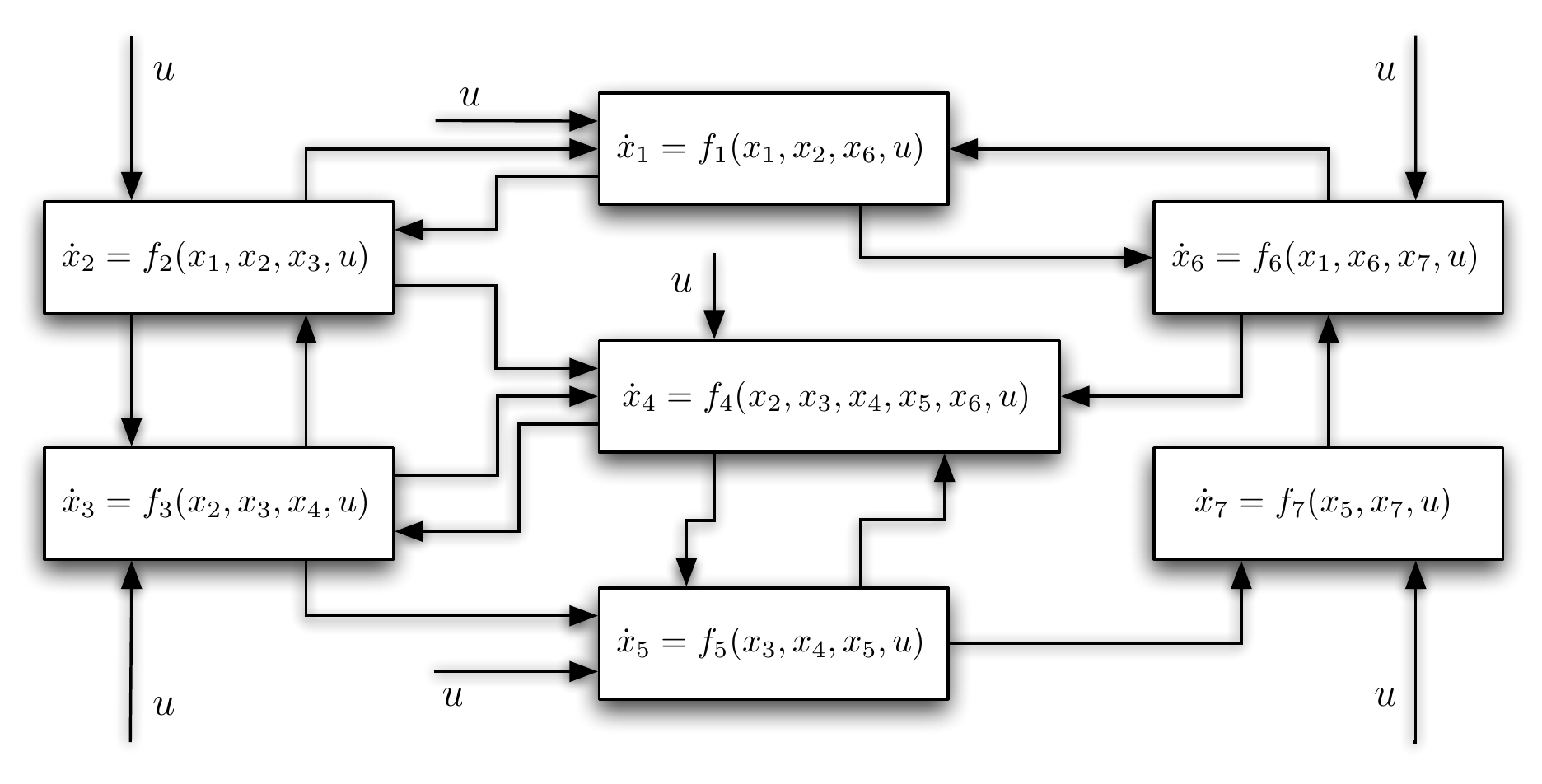}
\caption{%
  An example of a network of interconnected systems and the
  associated graph.}
\label{fig:networkfigure}
\end{figure}
Associated to such a network is a directed graph, with vertices
representing the subsystems and where the directed edges $(i,j)$
correspond to inputs going from system $j$ to system $i$, see
Figure \ref{fig:networkfigure}.  We will call the network strongly
connected if its interconnection graph has the same property.

For networks of the type that has been just described we wish to
construct Lyapunov functions as they are introduced now.

\subsection{Stability}
\label{subsec:stability}

An appropriate stability notion to study nonlinear systems with
inputs is input-to-state stability, introduced in \cite{Son89}.
The standard definition is as follows.

A forward complete system $\dot{x} = f(x,u)$ with $x\in \R^N, u\in
\R^M$ is called input-to-state stable if there are $\beta \in
\KL$, $\gamma\in \K$ such that for all initial conditions $x_0 \in
\R^{N}$ and all $u\in L^\infty(\R_+,\R^M)$ we have
\begin{equation}
\label{basiciss} \norm{x(t;x_0,u(\cdot))} \leq \beta(\|x_0\|,t) +
\gamma(\norm{u}_\infty) \,.
\end{equation}
It is known to be an equivalent requirement to ask for the
existence of an ISS Lyapunov function, \cite{SoW95}. These
functions can be chosen to be smooth. For our purposes, however,
it will be more convenient to have a broader class of functions
available for the construction of a Lyapunov function. Thus we
will call a function a {\em Lyapunov
function candidate}, if the following assumption is met.

\begin{assumption}
\label{A1}
The function $V: \R^N \to \R_+$ is continuous, proper and positive
definite and locally Lipschitz continuous on $\R^N \setminus \{0\}$.
\end{assumption}
Note that by Rademacher's Theorem (e.g., \cite[Theorem 5.8.6,
p.281]{Eva98}) locally Lipschitz continuous functions on $\R^N \setminus \{0\}$ are differentiable
almost everywhere in $\R^N$.

\begin{definition}
\label{def:Lipschitz-ISS-Lf}
We will call a function satisfying Assumption~\ref{A1} an
ISS Lyapunov function for $\dot{x} = f(x,u)$, if there exist
$\gamma\in\K$, and a positive definite function $\alpha$ such that
in all points of differentiability of $V$ we have
\begin{equation}
  \label{basicissLF}
  V(x) \geq  \gamma(\norm{u}) \quad \implies \quad \nabla V(x) f(x,u) \leq
  -\alpha(\|x\|)\,.
\end{equation}
\end{definition}
ISS and ISS Lyapunov functions are related in the expected manner:
\begin{theorem}
\label{thm:equivalence-ISS-and-nonsmooth-ISS Lyapunov-function}
A system is ISS if and only if it admits an ISS Lyapunov function in
the sense of Definition~\ref{def:Lipschitz-ISS-Lf}.
\end{theorem}

This has been proved for smooth ISS Lyapunov functions in the
literature \cite{SoW95}. So the hard converse statement is clear,
as it is even possible to find smooth ISS Lyapunov functions,
which satisfy Definition~\ref{def:Lipschitz-ISS-Lf}. The
sufficiency proof for the Lipschitz continuous case goes along the
lines presented in \cite{SoW95,SoW96} using the necessary tools
from nonsmooth analysis, cf.~\cite[Theorem. 6.3]{CLSW98}.

Merely continuous ISS Lyapunov functions have been studied in
\cite[Ch.~3]{Gru02b}, arising as viscosity supersolutions to
certain partial differential inequalities. Here we work with the
Clarke generalized gradient $\partial V(x)$ of $V$ at $x$. For
functions $V$ satisfying Assumption~\ref{A1} Clarke's generalized
gradient satisfies for $x\neq 0$ that
\begin{equation}
\label{Clarke-def}
\partial V(x) = \conv \left\{ \zeta \in \R^n \;:\; \exists x_k \to x : \nabla V(x_k) \text{ exists and } \nabla V(x_k) \to \zeta \right\} \,.
\end{equation}
An equivalent formulation to \eqref{basicissLF} is given by
\begin{equation}
  \label{basicissLFclarke}
  V(x) \geq  \gamma(\norm{u}) \quad \implies \quad \forall \zeta \in \partial V(x) \colon \langle \zeta, f(x,u) \rangle\leq
  -\alpha(\|x\|)\,.
\end{equation}
Note that \eqref{basicissLFclarke} is also applicable in points
where $V$ is not differentiable.

The gain $\gamma$ in \eqref{basiciss} is in general different from
the ISS Lyapunov gain in \eqref{basicissLF}.  
In the sequel we will always
assume that gains are of class $\Kinf$.

\subsection{Monotone aggregation}
\label{sec:monotone-aggregation}

In this paper we concentrate on the construction of ISS Lyapunov
functions for the interconnected system $\Sigma$. For a single
subsystem~\eqref{eq:3}, in a similar manner to \eqref{basicissLF},
we wish to quantify the combined effect of the inputs $x_j$,
$j\neq i$, and $u$ on the evolution of the state $x_i$. As we will
see in the examples given in Section~\ref{examples} it depends on
the system under consideration how this combined effect can be
expressed, through the sum of individual effects, using the
maximum of individual effects or by other means.  In order to be
able to give a general treatment of this we introduce the notion
of {\em monotone aggregation functions} (MAFs).

\begin{definition}
A continuous function $\mu : \Rnp \to \R_+$ is called a monotone
aggregation function if the following three properties hold
\begin{itemize}
\item[(M1)] positivity: $\mu(s) \geq 0$ for all $s \in \Rnp$ and
  $\mu(s) > 0$ if $s\gneqq0$;
\item[(M2)] strict increase\footnote{%
     Cf. Assumption~\eqref{standard-assumption-on-compatibility-MAF-and-Gamma},
     where for the purposes of this paper (M2) is further
     restricted.}: if $x<y$, then $\mu(x) < \mu(y)$;
\item[(M3)] unboundedness: if $\|x\|\to\infty$ then
  $\mu(x)\to\infty$.
\end{itemize}
The space of monotone aggregation functions is denoted by $\MAF_n$
and $\mu\in\MAF_n^m$ denotes a vector $\MAF$, i.e., $\mu_i\in\MAF_n$,
for $i=1,\ldots,m$.
\end{definition}

A direct consequence of (M2) and continuity is the weaker
monotonicity property
\begin{itemize}
 \it
\item[(M2')] monotonicity:
 $x\leq y \implies \mu(x) \leq \mu(y)$.
\end{itemize}
In \cite{07-R-DISS,R08-positivity} MAFs have additionally been
required to satisfy another property,
\begin{enumerate}
 \it
\item[(M4)] subadditivity: $\mu(x+y)\leq\mu(x)+\mu(y)$,
\end{enumerate}
which we do not need for the constructions provided in this paper,
since we take a different approach, see
Section~\ref{sec:reducible}.

Standard examples of monotone aggregation functions satisfying
(M1)---(M4) are
$$
\begin{aligned}
 \mu(s) = \sum_{i=1}^n s_i^{l}, \mbox{ where }l\geq1, \quad \text{ or
 }\quad \mu(s) = \max_{i=1,\ldots,n} s_i \quad \text{ or }\\
 \quad
 \mu(s_{1},s_{2},s_{3},s_{4}) = \max\{s_{1},s_{2}\} + \max\{s_{3},s_{4}\}
 \,.
\end{aligned}
$$
On the other hand, the following function is not a MAF, since (M1)
and (M3) are not satisfied; $\nu(s) = \prod_{i=1}^{n}s_{i}$.

Using this definition we can define a notion of ISS Lyapunov
function for systems with multiple inputs.
In this case $\Sigma_i$ in \eqref{eq:3} will have several gains
$\gamma_{ij}$ corresponding to the inputs $x_j$. For notational
simplicity, we will include the gain $\gamma_{ii}\equiv0$
throughout this paper.  The following definition requires only
Lipschitz continuity of the Lyapunov function.

\begin{definition}
\label{ISSdef}
Consider the interconnected system \eqref{eq:4} and assume that for
each subsystem $\Sigma_j$ there is a given function $V_j:\R^{N_j}
\to \R_+$ satisfying Assumption~\ref{A1}.

For $i=1,\ldots,n$ the function $V_i : \R^{N_i} \to \R_+$ is called
an ISS Lyapunov function for $\Sigma_i$, if there exist $\mu_i \in
\MAF_{n+1}, \gamma_{ij} \in \Kinf \cup \{ 0\}, j\neq i,
\gamma_{iu}\in \K \cup\{0\}$, and a positive definite function
$\alpha_i$ such that
at all points of differentiability of $V_{i}$
\begin{equation}
  \label{ISScond}
  \begin{split} V_i(x_i) \geq \mu_i\left( \gamma_{i1}(V_1(x_1)),\ldots,\gamma_{in}(V_n(x_n)),
      \gamma_{iu}(\norm{u})\right)  \\ \implies \quad \nabla V_i(x_i)
    f_i(x,u) \leq - \alpha_i(\norm{x_i}) \,.
  \end{split}
\end{equation}
The functions $\gamma_{ij}$ and $\gamma_{iu}$ are called ISS Lyapunov
gains.
\end{definition}

Several examples of ISS Lyapunov functions are given in the next
section.

Let us call $x_j$, $j\ne i$, the \emph{internal inputs} to
$\Sigma_i$ and $u$ the \emph{external input}.  Note that the role
of functions $\gamma_{ij}$ and $\gamma_{iu}$ is essentially to
indicate whether there is any influence of different inputs on the
corresponding state.  In case $f_i$ does not depend on $x_j$ there
is no influence of $x_j$ on the state of $\Sigma_i$. In this case
we define $\gamma_{ij}\equiv 0$,
in particular always $\gamma_{ii}\equiv 0$. This allows us to
collect the internal gains into a matrix
\begin{equation}
 \Gamma:=(\gamma_{ij})_{i,j=1,\dots,n}\,.\label{eq:5}
\end{equation}
If we add the external gains as the last column into this matrix
then we denote it by $\ol\Gamma$.  The function $\mu_i$ describes
how the internal and external gains interactively enter in a
common influence on $x_i$. The above definition motivates the
introduction of the following nonlinear map
\begin{equation}
 \label{Gammadef}
 \ol{\Gamma}_\mu :
 \R^{n+1}_+  \to    \Rnp,\quad
 \begin{bmatrix}
   s_1 \\ \vdots \\ s_n \\ r
 \end{bmatrix}
 \mapsto
 \begin{bmatrix}
   \mu_1(\gamma_{11}(s_1),\ldots,\gamma_{1n}(s_n), \gamma_{1u}(r)) \\ \vdots \\
   \mu_n(\gamma_{n1}(s_1),\ldots,\gamma_{nn}(s_n),\gamma_{nu}(r))
 \end{bmatrix}\,.
\end{equation}
Similarly we define $\Gamma_\mu(s):=\ol{\Gamma}_\mu(s,0)$.  The
matrices $\Gamma$ and $\ol\Gamma$ are from now on referred to as
\emph{gain matrices}, $\Gamma_\mu$ and $\ol\Gamma_\mu$ as
\emph{gain
 operators}.

\begin{remark}[general assumption]
 \label{genass}
 Given $\Gamma\in\Kinfnn$ and $\mu\in\MAF^{n}$, we will from now on assume
 that $\Gamma$ and $\mu$ are \emph{compatible} in the following sense:
 For each $i=1,\ldots,n$, let $I_{i}$ denote the set of indices
 corresponding to the nonzero entries in the $i$th row of
 $\Gamma$. Then it is understood that also the restriction of $\mu_{i}$
 to the indices $I_{i}$ satisfies (M2), i.e.,
 \begin{equation}
   \label{standard-assumption-on-compatibility-MAF-and-Gamma}
  \mu_{i}(x|_{I_{i}}) <
   \mu_{i}(y|_{I_{i}}) \quad\text{~if~}\quad
   x|_{I_{i}} < y|_{I_{i}}\,.
 \end{equation}
 In particular we assume that the function
 $$
 s\mapsto \mu(s_{1},\ldots,s_{n}, 0),\quad s\in\Rnp,
 $$
 for $\mu\in\MAF_{n+1}$ satisfies (M2).  Note that (M1) and (M3) are
 automatically satisfied.
\end{remark}

The examples in the next section show explicitly how the
introduced functions, matrices and operators may look like for
some particular cases. Clearly, the gain operators will have to
satisfy certain conditions if we want to be able to deduce
that~\eqref{eq:4} is ISS with respect to external inputs, see
Section~\ref{sec:lyapunov-functions}.

\section{Examples for monotone aggregation}
\label{examples} In this section we show how different MAFs may
appear in different applications, for further examples see
\cite{DRW-CDC08-applications}. We begin with a purely academic
example and discuss linear systems and neural networks later in
this section. Consider the system
 \begin{equation}
   \dot x=-x-2x^3+\frac{1}{2}(1+2x^2)u^2+\frac{1}{2}w\label{eq:6}
 \end{equation}
 where $x,u,w\in\R$. Take $V(x)=\frac{1}{2}x^2$ as a Lyapunov
 function candidate.
 It is easy to see that if $|x|\geq u^2$ and $|x|\geq |w|$ then
 $$
 \dot V \leq -x^2-2x^4+\frac{1}{2}x^2(1+2x^2)+\frac{1}{2}x^2 = -x^4 <0
 $$
 if $x\ne0$. The conditions $|x|\geq u^2$ and $|x|\geq |w|$
 translate into $|x|\geq\max\{u^2,|w|\}$ and in terms of $V$ this
 becomes
 $$
 V(x)\geq\max\{u^4/2, w^2/2\}
 \implies
 \dot V(x) \leq -x^4.
 $$
 This is a Lyapunov ISS estimate where the gains are aggregated using
 a maximum, i.e., in this case we can take
 $\mu(s_{1},s_{2})=\max\{s_{1},s_{2}\}$ and $\gamma_{u}(r)=r^4/2$ and
 $\gamma_{w}(r)=r^2/2$.

 Note that there is a certain arbitrariness in the choice of $\mu$ and $\gamma_{ij}$.
 \label{revision:discussion-of-nonuniqueness-of-mu-vs-gamma}
 In the example one could as well take $\gamma_{u}(r)=\gamma_{w}(r)=r$ and
 $\mu(s_{1},s_{2})=\max\{s_{1}^{4}/2,s_{2}^{2}/2\}$, giving exactly
 the same condition, but different gains and a different monotone
 aggregation function. At the end of the day the small gain condition
 comes down to mapping properties of $\Gamma_{\mu}$.
 Different choices of $\Gamma$ and $\mu$ may lead to the same operator $\Gamma_{\mu}$.
 However, as we will see at a later stage, certain choices of $\mu$ can be
 computationally more convenient than others. In particular, if we
 can choose $\mu=\max$, the task of checking the small gain condition
 reduces to checking a cycle condition, cf.\
 Section~\ref{sec:special-case:mu=max}.

\subsection{Linear systems}
\label{section:linear} Consider linear
interconnected systems
\begin{equation}
\label{linear-systems}
\Sigma_i:\;\dot x_i=A_i x_i+\sum_{j=1}^n \Delta_{ij}x_j+B_i
u_i,\quad i=1,\dots,n,
\end{equation}
with $x_i\in\R^{N_i},u_i\in\R^{M_i},$ and matrices
$A_i,B_i,\Delta_{ij}$ of appropriate dimensions. Each system
$\Sigma_i$ is ISS from
$(x_1^T,\dots,x_{i-1}^T,x_{i+1}^T,\dots,x_n^T,u_i^T)^T$ to $x_i$
if and only if $A_i$ is Hurwitz. It is known that $A_i$ is Hurwitz
if and only if for any given symmetric positive definite $Q_i$
there is a unique symmetric positive definite solution $P_i$ of
$A_i^T P_i + P_i A_i = - Q_i$, see, e.g., \cite[Cor.\ 3.3.47 and
Rem.\ 3.3.48,
p.284f]{HiP05}. Thus we choose the
Lyapunov function $V_i(x_i) = x_i^T P_i x_i$, where $P_i$ is the
solution corresponding to a symmetric positive definite $Q_i$. In
this case, along trajectories of the autonomous system
$$
\dot x_i=A_ix_i
$$
we have
$$
\dot V_i=x_i^T P_i A_ix_i+x_i^TA_i^T P_ix_i=-x_i^T Q_i x_i\leq
-c_i\|x_i\|^2
$$
for $c_i := \lambda_{min}(Q_i)>0$, the smallest eigenvalue of
$Q_i$. For system~\eqref{linear-systems} we obtain
\begin{align}
\nonumber
\dot V_i&=
2x_i^T P_i \Big(A_i x_i+\sum_{j\ne i} \Delta_{ij}x_j+B_i u_i\Big)
\\
\label{eq:7}
&\leq -c_i \|x_i\|^2 + 2 \|x_i\| \|P_i\| \Big( \sum_{j\ne i} \|\Delta_{ij}\|
\|x_j\| + \|B_i\| \|u_i\|\Big) \leq -\epsilon c_i \|x_i\|^2,
\end{align}
where the last inequality~\eqref{eq:7} is satisfied for a given
$0<\epsilon<1$ if
\begin{equation}
\label{eq:8}
\|x_i\|\geq \frac{2\|P_i\|}{c_i(1-\epsilon)}
\Big( \sum_{j\ne i} \|\Delta_{ij}\|
\|x_j\|  + \|B_i\| \|u\|\Big)
\end{equation}
with $u:=(u_1^T,\dots,u_n^T)^T$. To write this implication in the
form \eqref{ISScond} we note that $\lambda_{\min}(P_i)\|x_i\|^2\le
V_i(x_i)\le \lambda_{\max}(P_i)\|x_i\|^2$. Let us denote
$a_i^2=\lambda_{\min}(P_i)$, $b_i^2=\lambda_{\max}(P_i) =
\|P_i\|$, then the inequality \eqref{eq:8} is satisfied if
$$
\|P_{i}\|\cdot \|x_{i}\|^{2}\geq V_i(x_i)\ge \|P_i\|^3
\left(\frac{2}{c_i(1-\epsilon)}\right)^2
\left(
\sum_{j\ne i}
\frac{\|\Delta_{ij}\|}{a_j}
\sqrt{V_j(x_j)}  +
\|B_i\|
\|u\|
\right)^2\,.
$$
This way we see that the function $V_i$ is an ISS Lyapunov
function for $\Sigma_i$ with gains given by
$$
\gamma_{ij}(s)= \left(
\frac{2b_i^3}{c_i(1-\epsilon)}
\frac{\|\Delta_{ij}\|}{a_j}\right)\
\sqrt{s}
$$
for $i=1,\dots,n$, $i\ne j$, and
$$
\gamma_{iu}(s)=  \frac{2\|B_i\|b_i^3}{c_i(1-\epsilon)} \ s,
$$
for $i=1,\ldots,n$, and $s\geq0$.  Further we have
$$
\mu_i(s,r)=\left(\sum_{j=1}^{n}s_j+r\right)^2
$$
for $s\in\R^{n}_+$ and $r\in\R_+$. This $\mu_i$ satisfies (M1),
(M2), and (M3), but not (M4).  By defining $\gamma_{ii}\equiv 0$
for $i=1,\dots,n$ we can write
$$
\ol\Gamma=\left(
\begin{array}{ccccc}
  0 & \gamma_{12} & \cdots & \gamma_{1n} & \gamma_{1u}\\
  \gamma_{21} & \ddots & \cdots & \gamma_{2n} & \gamma_{2u}\\
  \vdots & & \ddots & \vdots &\vdots\\
  \gamma_{n1} & \cdots & \gamma_{n,n-1} & 0 & \gamma_{nu}\\
\end{array}\right)
$$
and have
\begin{equation}
  \ol\Gamma_{\mu}(s,r)=
  \begin{pmatrix}
    \Big(\frac{2b_1^3}{c_1(1-\epsilon)}\Big)^2\Big(\sum_{j\ne 1} \frac{\|\Delta_{1j}\|}{a_j}\sqrt{s_j}+\|B_1\|r\Big)^2\\
    \vdots\\
    \Big(\frac{2b_n^3}{c_n(1-\epsilon)}\Big)^2\Big(\sum_{j\ne n} \frac{\|\Delta_{nj}\|}{a_j}\sqrt{s_j}+\|B_n\|r\Big)^2\\
  \end{pmatrix}.
\label{eq:9}
\end{equation}
Interestingly, the choice of quadratic Lyapunov functions for the
subsystems naturally leads to a nonlinear mapping
$\ol{\Gamma}_\mu$ with a useful homogeneity property, see Proposition~\ref{prop:homogen}.

\subsection{Neural networks}
\label{section:neural networks}

As the next example consider a Cohen-Grossberg neural network as in \cite{WaZ02}. The
dynamics of each neuron is given by
\begin{equation}\label{eq:neural network}
{\mathrm{NN}_i}:~
\dot x_i(t)=-a_i(x_i(t))\Big(b_i(x_i(t))-\sum_{j=1}^n
t_{ij}s_j(x_j(t))+J_i\Big),
\end{equation}
$ i=1,\dots,n,\;n\ge2$, where $x_i$ denotes the state of the
$i$-th neuron and $a_i$ is a strictly positive amplification
function. 
As in \cite{WaZ02} we assume that the fixed point is shifted to the origin.
Then the function $b_i$ typically satisfies the sign condition
$b_i(x_i)x_i \geq 0$
and satisfies furthermore $|b_i(x_i)|>\tilde b_i(|x_i|)$ for some
$\tilde b_i\in\Kinf$. The activation function $s_i$ is typically
assumed to be sigmoid. The matrix $T=(t_{ij})_{i,j=1,\dots,n}$
describes the interconnection of neurons in the network and $J_i$
is a given constant input from outside. However for our
consideration we allow $J_i$ to be an arbitrary measurable
function in $L_\infty$.

In applications the matrix $T$
\label{revision:added-sentence-about-T} is usually the result of
training using some learning algorithm and appropriate training
data. The specifics depend on the type of network architecture and
learning algorithm chosen and on the particular application. Such
considerations are beyond the scope of the current paper. 
We simply assume that $T$ is given and concern
ourselves solely with stability considerations.

Note that for any sigmoid function $s_i$ there exists a
$\gamma_i\in\K$ such that $|s_i(x_i)|<\gamma_i(|x_i|)$. Following
\cite{WaZ02} we assume
$0<\underline\alpha_i<a_i(x_i)<\ol\alpha_i$,
$\underline\alpha_i,\ol\alpha_i\in\R$.

Recall the triangle inequality for $\Kinf$-functions: For any
$\gamma,\rho\in\Kinf$ and any $a,b\ge0$ it holds
$$\gamma(a+b)\le\gamma\circ(\id+\rho)(a)+\gamma\circ(\id+\rho^{-1})(b).$$

We claim that $V_i(x_i):=|x_i|$ is an ISS Lyapunov function for ${\mathrm{NN}}_i$ in \eqref{eq:neural network}.
Fix an arbitrary 
function
$\rho\in\Kinf$ and some $\varepsilon$ satisfying
$\underline\alpha_i>\varepsilon>0$.
Then by the triangle inequality we have
\begin{multline*}
 |x_i|>\tilde
 b_i^{-1}\circ(\id+\rho)\left(\frac{\overline\alpha_i}{\underline\alpha_i-\varepsilon}
   \sum_{j=1}^n|t_{ij}|\gamma_j(|x_j|)\right)
 +\tilde b_i^{-1}\circ(\id+\rho^{-1})\left(\frac{\overline\alpha_i}{\underline\alpha_i-\varepsilon}|J_i|\right)\\
\tilde
 b_i^{-1}\left(\frac{\overline\alpha_i}{\underline\alpha_i-\varepsilon}
   \bigg(\sum_{j=1}^n|t_{ij}|\gamma_j(|x_j|)+|J_i|\bigg)\right)\\
 \implies
 \dot V_i=-a_i(x_i)\Big(|b_i(x_i)|-\mbox{sign}\,x_i\sum_{j=1}^n t_{ij}s_j(x_j)+\mbox{sign}\,x_iJ_i\Big)
 <-\varepsilon|b_i(x_i)|\,.
\end{multline*}
In this case we have
$$
\mu_i(s,r)\\=\tilde b_i^{-1}\circ(\id+\rho)(s_1+\dots+s_n)+\tilde
b_i^{-1}\circ(\id+\rho^{-1})(r)
$$
additive with respect to the external input and
$$
\gamma_{ij}=\frac{\overline\alpha_i|t_{ij}|}{\underline\alpha_i-\varepsilon}\gamma_j(|x_j|),\quad
\gamma_{iu}=\frac{\overline\alpha_i\id}{\underline\alpha_i-\varepsilon}.
$$
The MAF $\mu_i$ satisfies (M1), (M2), and (M3).  It satisfies (M4)
if and only if $(\tilde b_i)^{-1}$ is subadditive.

\section{Monotone Operators and generalized small gain conditions}
\label{sec:monot-oper-gener}

In Section~\ref{sec:monotone-aggregation} we saw that in the ISS
context the mutual influence between subsystems~\eqref{eq:3} and
the influence from external inputs to the subsystems can be
quantified by the gain matrices $\Gamma$ and $\ol\Gamma$ and gain
operators $\Gamma_{\mu}$ and $\ol\Gamma_{\mu}$.  The
interconnection structure of the subsystems naturally leads to a
weighted, directed graph, where the weights are the nonlinear gain
functions, and the vertices are the subsystems. There is an edge
from the vertex $i$ to the vertex $j$ if and only if there is an
influence of the state $x_i$ on the state $x_j$, i.e., there is a
nonzero gain $\gamma_{ji}$.

Connectedness properties of the interconnection graph together
with mapping properties of the gain operators will yield a
generalized small gain condition.  In essence we need a nonlinear
version of a Perron vector for the construction of a Lyapunov
function for the interconnected system.  This will be made
rigorous in the sequel. But first we introduce some further
notation.

The adjacency matrix $A_\Gamma=(a_{ij})$ of a matrix
$\Gamma\in\Kinfnn$ is defined by $a_{ij}=0$ if
$\gamma_{ij}\equiv0$ and $a_{ij}=1$ otherwise. Then
$A_\Gamma=(a_{ij})$ is also the adjacency matrix of the graph
representing an interconnection.

We say that a matrix $\Gamma$ is \emph{primitive, irreducible} or
\emph{reducible} if and only if $A_\Gamma$ is primitive,
irreducible or reducible, respectively. Recall (and see
\cite{BeP79} for more on this subject) that a nonnegative matrix
$A$ is
\begin{itemize}
\label{revision:recalling-irreducibility} \item \emph{primitive}
if there exists a $k\geq1$ such that $A^{k}$ is
 positive;
\item \emph{irreducible} if for every pair $(i,j)$ there exists a
 $k\geq1$ such that the $(i,j)$th entry of $A^{k}$ is positive;
 obviously, primitivity implies irreducibility;
\item \emph{reducible} if it is not irreducible.
\end{itemize}
A network or a graph is strongly connected if and only if the
associated adjacency matrix is irreducible, see also~\cite{BeP79}.

For $\Kinf$ functions $\alpha_1,\dots,\alpha_n$ we define a
diagonal operator $D:\R^n_+\rightarrow\R^n_+$ by
\begin{equation}\label{eq:D}
D(s):=(s_1+\alpha_1(s_1),\dots,s_n+\alpha_n(s_n))^T, \quad
s\in\R^n_+.
\end{equation}

For an operator $T:\Rnp\to\Rnp$, the condition $T\ngeq\id$ means
that for all $s\ne0$, $T(s)\ngeq s$. In words, at least one
component of $T(s)$ has to be strictly less than the corresponding
component of $s$.

\begin{definition}[Small gain conditions]
\label{def:small-gain-conditions}
Let a gain matrix $\Gamma$ and a monotone aggregation $\mu$ be
given. The operator
$\Gamma_\mu$ is said to satisfy the small gain condition~\eqref{eq:small-gain-condition}, if
\begin{equation}
  \label{eq:small-gain-condition}
  \tag{SGC}
  \Gamma_{\mu}\not\ge\id,
\end{equation}
Furthermore, $\Gamma_\mu$ satisfies the strong small gain
condition~\eqref{eq:strong-small-gain-condition}, if there exists a
$D$ as in \eqref{eq:D} such that
\begin{equation}
  \label{eq:strong-small-gain-condition}
  \tag{sSGC}
  D\circ\Gamma_{\mu}\not\ge\id \,.
\end{equation}
\end{definition}
It is not difficult to see that
\eqref{eq:strong-small-gain-condition} can equivalently be stated
as
\begin{equation}
  \label{eq:strong-small-gain-condition-postfix-D}
  \tag{sSGC'}
  \Gamma_{\mu}\circ D\ngeq\id.
\end{equation}
Also for \eqref{eq:strong-small-gain-condition} or
\eqref{eq:strong-small-gain-condition-postfix-D} to hold it is
sufficient to assume that the function $\alpha_1,\ldots,\alpha_n$
are all identical. This can be seen by defining $\alpha(s) :=
\min_i\alpha_i(s)$.  We abbreviate this in writing
$D=\diag(\id+\alpha)$ for some $\alpha\in\Kinf$.

For maps $T:\Rnp\to\Rnp$ we define the following sets:
\begin{gather*}
\Omega(T):=\{s\in\Rnp:T(s)< s\}=\bigcap_{i=1}^n
\Omega_i(T), \text{ where}\\
\Omega_i(T):=\{s\in\Rnp:T(s)_i <  s_i\} \,.
\end{gather*}
If no confusion arises we will omit the reference to $T$.
Topological properties of the introduced sets are related to the
small gain conditions \eqref{eq:small-gain-condition}, cf.\ also
\cite{DRW05-cdcecc,DRW-mcss,R08-positivity}.  They will be used in
the next section for the construction of an ISS Lyapunov function
for the interconnection.

\section{Lyapunov functions}
\label{sec:lyapunov-functions}

In this section we present the two main results of the paper. The
first is a topological result on the existence of a jointly
unbounded path in the set $\Omega$, provided that $\Gamma_\mu$
satisfies the small gain condition. This path will be crucial in
the construction of a Lyapunov function, which is the second main
result of this section.

\begin{definition}
\label{def:omega-path} A continuous path $\sigma \in \Kinf^n$ will
be called an {\em $\Omega$-path with
respect to $\Gamma_\mu$} if
\begin{enumerate}
\item for each $i$, the function $\sigma_i^{-1}$ is locally Lipschitz
  continuous on $(0,\infty)$;
\item for every compact set $K\subset (0,\infty)$ there are constants
  $0<c<C$ such that for all $i=1,\ldots,n$ and all points of differentiability of
  $\sigma_i^{-1}$  we have
  \begin{equation}
    \label{sigma-bounds}
    0<c \leq (\sigma_i^{-1})'(r) \leq C\,,\quad \forall r\in K;
  \end{equation}
\item $\sigma(r) \in \Omega(\Gamma_\mu)$ for all $r>0$, i.e.
  \begin{equation}
    \label{sigma}
    \Gamma_\mu(\sigma(r)) < \sigma(r) \,,\quad \forall r> 0\,.
  \end{equation}
\end{enumerate}
\end{definition}

Now we can state the first of our two main results, which regards
the existence of $\Omega$-paths.

\begin{theorem}
\label{path}
Let $\Gamma \in \Kinfnn$ be a gain matrix and $\mu \in \MAF_n^n$.
Assume that one of the following assumptions is satisfied
\begin{enumerate}
\item $\Gamma_\mu$ is linear and the spectral radius of
  $\Gamma_{\mu}$ is less than one;
\item $\Gamma$ is irreducible and $\Gamma_\mu\ngeq\id$;
\item\label{item:3} $\mu=\max$ and $\Gamma_\mu\ngeq\id$;
\item alternatively assume that $\Gamma_{\mu}$ is bounded, i.e.,
  $\Gamma\in((\K\setminus\Kinf)\cup\{0\})^{n\times n}$, and satisfies
  $\Gamma_\mu\ngeq\id$.
\end{enumerate}
 Then there exists an $\Omega$-path $\sigma$
with respect to $\Gamma_\mu$.
\end{theorem}

We will postpone the proof of this rather topological result to
Section~\ref{sec:pathconstruction} and reap the fruits of
Theorem~\ref{path} first.  Note, however, that for~\ref{item:3}
there exists a ``$\text{cycle gain} < \id$''-type equivalent
formulation, cf.\ Theorem~\ref{thm:mu=max} 
and see \cite{Pot03,Tee05,DRW-mcss,KarafyllisJiang:5:A-Vector-Small-Gain-Theorem-for-General-:}.

In addition to the above result, the existence of $\Omega$-paths
can also be asserted for reducible $\Gamma$ and $\Gamma$ with
mixed, bounded and unbounded, class $\K$ entries, see
Theorem~\ref{thm:monotone-path-reducible-case} and
Proposition~\ref{prop:partly-bounded-Gamma}, respectively.

\begin{theorem}
\label{LFtheo}
Consider the interconnected system $\Sigma$ given by \eqref{eq:3},
\eqref{eq:4} where each of the subsystems $\Sigma_i$ has an
ISS Lyapunov function $V_i$, the corresponding gain matrix is given
by \eqref{eq:5}, and $\mu=(\mu_1,\ldots,\mu_n)^T$ is given
by~\eqref{ISScond}. Assume there are an $\Omega$-path $\sigma$ with
respect to $\Gamma_\mu$ and a function $\phi \in \Kinf$ such that
\begin{equation} \label{generalcond}
  \ol{\Gamma}_\mu(\sigma(r),\phi(r)) < \sigma(r) \,,\quad
  \forall \ r> 0
\end{equation}
is satisfied, then an ISS Lyapunov function for the overall system
is given by
\begin{equation}
  \label{Vdef}
  V(x) = \max_{i=1,\ldots,n} \sigma_i^{-1} (V_i(x_i))\,.
\end{equation}
In particular, for all points of differentiability of
$V$ we have the implication
\begin{equation}
  V(x) \geq \max \{
  \phi^{-1}(\gamma_{iu}(\|u\|)) \;|\; i=1,\ldots n \}  \implies \nabla
  V(x) f(x,u) \leq -\alpha(\|x\|)\,,
\end{equation}
 where $\alpha$ is a suitable positive definite function.
\end{theorem}

Note that by construction the Lyapunov function $V$ is not smooth,
even if the functions $V_i$ for the subsystems are. This is why it
is appropriate in this framework to consider Lipschitz continuous
Lyapunov functions, which are differentiable almost everywhere.

\begin{proof}
We will show the assertion in the Clarke gradient sense. For $x=0$ there is nothing to show. So let
$0\neq x = (x_1^T,\ldots,x_n^T)^T$. Denote by $I$ the set of indices $i$ for which
\begin{equation}
\label{eq:firststep}
  V(x) = \sigma_i^{-1}(V_i(x_i))
\geq
 \max_{j\neq i} \sigma_j^{-1} (V_j(x_j)) \,.
\end{equation}
Then $x_i\neq 0$, for $i\in I$.
Also as $V$ is obtained through maximization we have
because of \cite[p.83]{CLSW98} that
\begin{equation}
\label{convrep}
\partial V(x)\subset \conv \left\{ \bigcup_{i\in
    I} \partial[\sigma_i^{-1}\circ V_i\circ\pi_i](x)\right\}\,.
\end{equation}

Fix $i\in I$ and assume without loss of generality $i=1$. Then if we
assume $V(x) \geq \max_{i=1,\ldots,n} \{ \phi^{-1}(\gamma_{iu}(\| u
\|))\}$ it follows in particular that $\gamma_{1u}(\norm{u}) \leq
\phi(V(x))$.  Using the abbreviation $r:= V(x)$, denoting the first
component of $\overline{\Gamma}_\mu$ by $\overline{\Gamma}_{\mu,1}$ and
using assumption \eqref{generalcond} we have
 \begin{equation*}
 \begin{aligned}
  V_1(x_1) =&\;  \sigma_1 (r) > \overline\Gamma_{\mu,1}(\sigma(r),\phi(r))\\
  & = \mu_1\left[\gamma_{11}(\sigma_1(r)),\ldots,\gamma_{1n}(\sigma_n(r)),\phi(r)\right] \\
  &\geq \mu_1\left[\gamma_{11}(\sigma_1(r)),\ldots,\gamma_{1n}(\sigma_n(r)), \gamma_{1u}(\norm{u})\right] \\
  &  = \mu_1\left[
  \gamma_{11}\circ \sigma_1 \circ
    \sigma_1^{-1} (V_1(x_1))
    ,  \ldots,
  \gamma_{1n}\circ \sigma_n \circ \sigma_1^{-1}
  (V_1(x_1)),\gamma_{1u}(\norm{u})\right]   \\
  & \geq \mu_1 \left[ \gamma_{11}\circ V_1(x_1),\ldots, \gamma_{1n}\circ
  V_n(x_n), \gamma_{1u}(\norm{u})\right] \,,
  \end{aligned}
\end{equation*}
where we have used \eqref{eq:firststep} and (M2') in the last
inequality.  Thus the ISS condition \eqref{ISScond} is applicable and
we have for all $\zeta\in \partial V_1(x_1)$ that
\begin{equation}
  \label{eq:10}
    \langle \zeta , f_1(x,u)\rangle
    \leq - {\alpha}_1(\norm{x_1})\,\,.
\end{equation}
By the chain rule for Lipschitz continuous functions \cite[Theorem 2.5]{CLSW98} we have
\[ \partial (\sigma_i^{-1}\circ V_i )(x_i) \subset \{ c \zeta \;:\; c
\in \partial \sigma_i^{-1}(y) \,, y= V_i(x_i)\,, \zeta \in \partial
V_i(x_i) \} \,.\] Note that in the previous equation the number $c$
is bounded away from zero because of \eqref{sigma-bounds}. We set for
$\rho>0$
\[
\tilde{\alpha}_i(\rho) := c_{\rho,i} \, \alpha_i(\rho) > 0 \,,
\]
where $c_{\rho,i}$ is the constant corresponding to the set $K:=
\{x_i\in \R^{N_i}\ : \ \rho/2 \leq \|x_i\| \leq 2\rho \}$ given by
\eqref{sigma-bounds} in the definition of an $\Omega$-path.  With the
convention $x= (x_1^T,\ldots,x_n^T)^T$ we now define for $r>0$
\[
\alpha(r) = \min \{ \tilde \alpha_i(\|x_i\|) \;|\; \|x\|= r, V(x) =
\sigma_i^{-1}(V_i(x_i))) \} >0 \,.
\]
Here we have used that for a given $r>0$ and $\|x\|=r$ the norm of
$\|x_i\|$ such that $V(x) = \sigma_i^{-1}(V_i(x_i)))$ is bounded away
from $0$.

It now follows from \eqref{eq:10} that if $V(x) \geq
\max_{i=1,\ldots,n} \{ \phi^{-1}(\gamma_{iu}(\| u \|))\}$, then we
have for all $\zeta\in \partial \left[\sigma_1^{-1} \circ V_1\right]
(x_1)$ that
\begin{equation}
  \label{eq:8b}
    \langle \zeta , f_1(x,u)\rangle
    \leq - {\alpha}(\norm{x})\,\,.
\end{equation}
In particular, the right hand side depends on $x$ 
and not only on $x_1$.  The
same argument applies for all $i\in I$. Now for any $\zeta
\in \partial V(x)$ we have by \eqref{convrep} that $\zeta =
\sum_{i\in I} \lambda_i c_i \zeta_i$ for suitable $\lambda_i \geq 0,
\sum_{i\in I} \lambda_i=1$ and with $\zeta_i \in \partial (V_i\circ
\pi_i)(x)$ and $c_i \in \partial \sigma_i^{-1} (V_i(x_i))$. It
follows that
\begin{equation*}
  \begin{aligned}
    \langle \zeta, f(x,u) \rangle = \sum_{i\in I} \lambda_i \langle c_i
    \zeta_i , f(x,u) \rangle
    = \sum_{i\in I} \lambda_i \langle c_i
    \pi_{i}(\zeta_i ), f_{i}(x,u) \rangle\\
    \leq - \sum_{i\in I} \lambda_i {\alpha}(\norm{x}) = - \alpha(\norm{x}) \,.
  \end{aligned}
\end{equation*}
This shows the assertion.
\end{proof}

In the absence of external inputs, ISS is the same as 0-GAS (cf.\
\cite{SoW95b,SoW95,SoW96}). We note the following consequence in the 
case that only global asymptotic stability is of interest.

\begin{corollary}[0-GAS for strongly interconnected networks]
 \label{cor:0-GAS}
 In the setting of Theorem~\ref{LFtheo}, assume that the external
 inputs satisfy $u\equiv 0$ and that the network of interconnected
 systems is strongly connected. If $\Gamma_{\mu}\ngeq\id$ then the
 network is 0-GAS.
\end{corollary}
\begin{proof}
 By Theorem~\ref{path}~(ii) there exists an $\Omega$-path and a
 nonsmooth Lyapunov for the network is given by~\eqref{Vdef}, hence
 the origin of the externally unforced composite system is GAS.
\end{proof}

\begin{remark}
At first sight it might seem that the previous corollary
is stronger than \cite[Cor. 2.1]{JTP94}, as no robustness
term $D$ is needed in the assumptions. 
However, the result here is formulated for Lyapunov
functions whereas the result in \cite{JTP94} is based on the
trajectory formulation of ISS in summation form. The proof in the 
trajectory version essentially requires
bounds on $(\id-\Gamma_{\mu})^{-1}$, which relies heavily on $D$
unless $\mu=\max$, \cite{JTP94,DRW-mcss,07-R-DISS}.
In contrast, for $0$-GAS 
the $D$ is not needed in the Lyapunov setting, because 
for irreducible $\Gamma$ it is
possible to construct the path $\sigma$ without $D$ by Theorem~\ref{path}~(ii). 
\end{remark}

We now specialize the Theorem~\ref{LFtheo} to particular cases of
interest. Namely, when the gain with respect to the external input
$u$ enters the ISS condition (i) additively, (ii) via maximization
and (iii) as a factor.

\begin{corollary}[Additive gain of external input $\mathbf{u}$]
\label{sumcor}
Consider the interconnected system $\Sigma$ given by \eqref{eq:3},
\eqref{eq:4} where each of the subsystems $\Sigma_i$ has an
ISS Lyapunov function $V_i$ and the corresponding gain matrix is given
by \eqref{Gammadef}. Assume that the ISS condition is additive in the
gain of $u$, that is,
\begin{equation}
  \ol{\Gamma}_\mu(V_1(x_1),\ldots, V_n(x_n),\norm{u}) =
  \Gamma_\mu(V_1(x_1),\ldots, V_n(x_n)) + \gamma_u(\norm{u}) \,,
\end{equation}
where
$\gamma_u(\|u\|)=(\gamma_{1u}(\|u\|),\ldots,\gamma_{nu}(\|u\|))^T$.
If $\Gamma_\mu$ is irreducible and if there exists an $\alpha \in \Kinf$ such that for
$D=\diag(\id + \alpha)$ the gain operator $\Gamma_\mu$ satisfies the
strong small gain condition
\[ D \circ \Gamma_\mu(s) \not \geq s \] then the interconnected system
is ISS and an ISS Lyapunov function is given by \eqref{Vdef}, where
$\sigma \in \Kinf^n$ is an arbitrary $\Omega$-path with respect to
$D\circ \Gamma_\mu$.
\end{corollary}

\begin{proof}
By Theorem~\ref{path} an $\Omega(D\circ \Gamma_\mu)$-path $\sigma$
exists.
Observe that by irreducibility, (M1), and (M3) it follows that
$\Gamma_\mu(\sigma)$ is unbounded in all
components. Let
$\phi \in \Kinf$ be such that for all $r\geq 0$
\[ \min_{i=1,\ldots, n} \{ \alpha( \Gamma_{\mu,i}(\sigma(r))) \}\geq
\max_{i=1,\ldots, n} \{\gamma_{iu}(\phi(r)) \}\,.\] Note that this
is possible, because on the left we take the minimum of a finite
number of $\Kinf$ functions. Then we have for all $r>0$,
$i=1,\ldots,n$ that
\[ \sigma_i(r) > D\circ\Gamma_{\mu,i}(\sigma(r)) =
\Gamma_{\mu,i}(\sigma(r)) + \alpha(\Gamma_{\mu,i}(\sigma(r)) )
\geq\Gamma_{\mu,i}(\sigma(r)) + \gamma_{iu}(\phi(r)) \,. \] Thus
$\sigma(r) > \ol{\Gamma}_\mu(\sigma(r),\phi(r))$ and the assertion
follows from Theorem~\ref{LFtheo}.
\end{proof}

\begin{corollary}[Maximization w.r.t. external gain]
\label{maxcor}
Consider the interconnected system $\Sigma$ given by \eqref{eq:3},
\eqref{eq:4} where each of the subsystems $\Sigma_i$ has an
ISS Lyapunov function $V_i$ and the corresponding gain matrix is
given by \eqref{Gammadef}. Assume that $u$ enters the ISS condition
via maximization, that is,
\begin{equation}
  \ol{\Gamma}_\mu(V_1(x_1),\ldots, V_n(x_n),\norm{u}) =
  \max\left\{ \Gamma_\mu(V_1(x_1),\ldots, V_n(x_n)),
    \gamma_u(\norm{u}) \right\} \,,
\end{equation}
where
$\gamma_u(\|u\|)=(\gamma_{1u}(\|u\|),\ldots,\gamma_{nu}(\|u\|))^T$.
Then, if $\Gamma_\mu$ is irreducible and satisfies the small gain condition
\[ \Gamma_\mu(s) \not \geq s \] the interconnected system is ISS and
an ISS Lyapunov function is given by \eqref{Vdef}, where $\sigma \in
\Kinf^n$ is an arbitrary $\Omega$-path with respect to
$\Gamma_\mu$ and $\phi$
is a $\Kinf$ function with the property
 \begin{equation}
\gamma_{iu}\circ \phi(r) \leq \Gamma_{\mu,i}(\sigma(r))  \,,\quad i=1,\ldots,n,
   \label{eq:11}
 \end{equation}
where $\Gamma_{\mu,i}$ denotes the $i$-th row of $\Gamma_\mu$.
\end{corollary}

\begin{proof}
 By Theorem~\ref{path} an $\Omega(\Gamma_\mu)$-path $\sigma$ exists.
 Note that by irreducibility, (M1), and (M3) it follows that
 $\Gamma_\mu(\sigma)$ is unbounded in all components.  Hence $\phi
 \in \Kinf$ satisfying~\eqref{eq:11} exists and we obtain
\[ \sigma(r) > \max \left\{ \ \Gamma_\mu(\sigma(r)), \gamma_u(\phi(r))
   \right\} \,.
\]
This is \eqref{generalcond} for the
case of maximization of gains in $u$. The claim follows from
Theorem~\ref{LFtheo}.
\end{proof}

In the next result observe that (M3) is not always necessary for
the $u$-component of $\mu$.

\begin{corollary}[Separation in gains]
\label{multcor}
Consider the interconnected system $\Sigma$ given by \eqref{eq:3},
\eqref{eq:4} where each of the subsystems $\Sigma_i$ has an ISS
Lyapunov function $V_i$ and the corresponding gain matrix $\Gamma$ is
given by \eqref{Gammadef}. Assume that $\Gamma$ is irreducible and
that the gains in the ISS condition are separated, that is, there exist
$\mu\in\MAF_{n}^{n}$, $c\in \R, c>0$, and $\gamma_u\in \Kinf$ such
that
\begin{equation}\label{eq:separation}
  \ol{\Gamma}_\mu(V_1(x_1),\ldots, V_n(x_n),\norm{u}) =
  \left(c+\gamma_u(\norm{u})\right)\ \Gamma_\mu(V_1(x_1),\ldots, V_n(x_n)) \,.
\end{equation}
If there exists an $\alpha \in \Kinf$ such
that for $D=\diag(c\cdot \id  + \id \cdot \alpha)$ the gain operator $\Gamma_\mu$
satisfies the strong small gain condition
\[ D \circ \Gamma_\mu(s) \not \geq s \] then the interconnected
system is ISS and an ISS Lyapunov function is given by \eqref{Vdef},
where $\sigma \in \Kinf^n$ is an arbitrary $\Omega$-path with
respect to $D \circ \Gamma_\mu(s)$.
\end{corollary}

\begin{proof}
If $\Gamma_\mu$ is irreducible, then also $D\circ \Gamma_\mu$ is irreducible and so
by Theorem~\ref{path}~(ii) an $\Omega(D\circ \Gamma_\mu)$-path $\sigma$
exists.
Let $\phi \in \Kinf$ be such that for all $r\geq 0$
\[
\phi(r) \leq \min_{i=1,\ldots,n} \{ \gamma_u^{-1} \circ \alpha \circ \Gamma_{\mu,i}(\sigma(r)) \}\,,
\]
where as in the previous corollaries we appeal to irreducibility,
(M1), and (M3).
Then for each $i$ we have
\[\sigma_i(r) > \Gamma_{\mu,i}(\sigma(r)) (c+ \alpha(\Gamma_{\mu,i}(\sigma(r))))
\geq
\Gamma_{\mu,i}(\sigma(r)) (c+\gamma_u\circ\phi(r))\] and hence
\[\sigma(r) > (c+ \gamma_u(\phi(r)))\Gamma_\mu(\sigma(r)) = \overline{\Gamma}_\mu(\sigma(r),\phi(r)) \]
and the assertion follows from \eqref{eq:separation} and
Theorem~\ref{LFtheo}.
\end{proof}

\section{The reducible case and scaling}
\label{sec:reducible}

The results that have been obtained so far concern mostly strongly
connected networks, that is, networks with an irreducible gain
operator. Already in \cite{SoT95} it has been shown that cascades
of ISS systems are ISS. Cascades are a special case of networks
where the gain matrix is reducible. In this section we briefly
explain how a Lyapunov function for a network that is not strongly 
connected may
be constructed based on the construction for the strongly
connected components of the network. Another approach would be to
construct the $\Omega$-path for reducible operators $\Gamma_{\mu}$
as has been done in \cite{R08-positivity} using assumption (M4). 

It is well known, that
if the network is not strongly connected, or equivalently if the
gain matrix $\Gamma$ is reducible, then $\Gamma$ may be brought in
upper block triangular form via a permutation of the vertices of
the network as in the nonnegative matrix case
\cite{BeP79,DRW-mcss}. After this transformation
$\overline{\Gamma}$ is of the form
\begin{equation}
\label{eq:reducedform} \overline{\Gamma} = \begin{bmatrix}
\Upsilon_{11} & \Upsilon_{12} & \ldots & \Upsilon_{1d} & \Upsilon_{1u} \\
0 & \Upsilon_{22} & \ldots & \Upsilon_{2d} & \Upsilon_{2u}\\
\vdots & & \ddots & &&\\
0 & \ldots & 0 & \Upsilon_{dd}& \Upsilon_{du}
\end{bmatrix} \,,
\end{equation}
where each of the blocks  on the diagonal $\Upsilon_{jj}\in (\Kinf
\cup \{ 0 \})^{d_j\times d_j}$ , $j=1,\ldots,d$, is either
irreducible or $0$. Let $q_j = \sum_{l=1}^{j-1} d_l$, with the
convention that $q_1=0$. We denote the states corresponding to the
strongly connected components by
\[ z_j^T = \begin{bmatrix}
x_{q_j + 1}^T & x_{q_j + 2}^T & \ldots & x_{q_{j+1}}^T
\end{bmatrix}\,.\]
We will show that in order to obtain an overall ISS Lyapunov
function it is sufficient to construct ISS Lyapunov functions for
each of the irreducible blocks (where the respective states with
higher indices are treated as inputs). The desired result is an
iterative application of the following observation.
\begin{lemma}
\label{lem:reduciblestep1} Let a gain matrix $\overline{\Gamma}
\in \left(\Kinf \cup \{ 0 \} \right)^{2\times 3}$ be given by
\begin{equation}
\label{eq:reduciblestep1} \overline{\Gamma} = \begin{bmatrix} 0 &
\gamma_{12} & \gamma_{1u} \\ 0 & 0 & \gamma_{2u}
\end{bmatrix} \,,
\end{equation}
and let $\overline{\Gamma}_\mu$ be defined by  $\mu \in \MAF_3^2$.
Then there exist an $\Omega$-path $\sigma$ and  $\phi \in \Kinf$
such that \eqref{generalcond} holds.
\end{lemma}

\begin{proof}
By construction the maps $\eta_1: r\mapsto \mu_1(\gamma_{12}(r),
\gamma_{1u}(r))$ and $\eta_2 : r \mapsto \mu_2(\gamma_{12}(u))$
are in $\Kinf$. Choose a $\Kinf$-function $\tilde \eta_1 \geq
\eta_1$, such that $\tilde \eta_1$ satisfies the conditions (i)
and (ii) in Definition~\ref{def:omega-path}. Define $\sigma(r) =
\begin{bmatrix} 2 \tilde \eta_1(r) & r \end{bmatrix}^T$ and
$\phi(r) := \min \{ r, \eta_2^{-1}(r/2)\}$. Then it is a
straightforward calculation to check that the assertion holds.
\end{proof}

The result is now as follows.

\begin{proposition}
 \label{prop:reducible}
 Consider a simply connected interconnected system $\Sigma$ given by
 \eqref{eq:3}, \eqref{eq:4} where each of the subsystems $\Sigma_i$
 has an ISS Lyapunov function $V_i$, the corresponding gain matrix is
 given by \eqref{eq:5}, and $\mu=(\mu_1,\ldots,\mu_n)^T$ is given
 by~\eqref{ISScond}.  Assume that the gain matrix
 $\overline{\Gamma}$ is in the reduced form \eqref{eq:reducedform}.
 If for each $j=1,\ldots, d-1$ there exists an ISS Lyapunov function
 $W_j$ for the state $z_j$ with respect to the inputs $z_{j+1},
 \ldots, z_d, u$ then there exists an ISS Lyapunov function $V$ for
 the state $x$ with respect to the input $u$.
\end{proposition}

\begin{proof}
 By assumption for each $j=1,\ldots, d-1$ there exist gain functions
 $\chi_{jk} \in \Kinf$ and $\chi_{ju} \in \Kinf$ 
 and MAFs $\tilde{\mu}_j$
 such that
 \begin{multline*}
     W_j(z_j) \geq \tilde \mu_j(\chi_{jj+1}(W_{j+1}(z_{j+1})),\ldots,
     \chi_{jd}(W_d(z_d)), \chi_{ju}(\|u\|)) \\ \Longrightarrow  \nabla
     W_j(z_j) f_j(z_j, z_{j+1},\ldots, z_d,u) < - \tilde
     \alpha_j(\|z_j\|)\,.
   \end{multline*}
We now argue by induction. If $d=1$, there is nothing to show. If
the result is shown for $d-1$ blocks, consider a gain matrix as in
\eqref{eq:reducedform}. By assumption there exists an ISS Lyapunov
function $V_{d-1}$ such that
\begin{multline*}
 V_{d-1} (z_{d-1}) \geq \mu_1( \gamma_{12} (V_d(z_d)), \gamma_{1u}
 (\|u\|)) \\ \Longrightarrow \nabla V_{d-1}(z_{d-1})
 f_{d-1}(z_{d-1},z_d,u) \leq - \alpha_{d-1} (\| z_{d-1}\|) \,.
\end{multline*}
As the remaining part has only external inputs, we see that
$\overline{\Gamma}$ is of the form \eqref{eq:reduciblestep1} and
so Lemma~\ref{lem:reduciblestep1} is applicable. This shows that
the assumptions of Theorem~\ref{LFtheo} are met and so a Lyapunov
function for the overall system is given by \eqref{Vdef}.
\end{proof}

It is easy to see that the assumption ${\Gamma}_\mu \not \geq \id$
(or ${\Gamma}_\mu \circ D\not \geq \id$) is equivalent to the
requirement that the blocks $\Upsilon_{jj}$ on the diagonal
satisfy the (strong) small gain
condition~\eqref{eq:small-gain-condition}/\eqref{eq:strong-small-gain-condition}.
Thus we immediately obtain the following statements.

\begin{corollary}[Summation of gains]
 \label{cor:addition_of_gains}
 Consider the interconnected system $\Sigma$ given by \eqref{eq:3},
 \eqref{eq:4} where each of the subsystems $\Sigma_i$ has an ISS
 Lyapunov function $V_i$ and the corresponding gain matrix is given
 by \eqref{Gammadef}. Assume that the ISS condition is additive in
 the gains, that is,
\begin{equation}\label{eq:additive-external-gain}
  \ol{\Gamma}_{\mu,i}(V_1(x_1),\ldots, V_n(x_n),\norm{u}) =
  \sum_{j=1}^{n}\gamma_{ij}(V_j(x_j)) + \gamma_{iu}(\norm{u}) \,.
\end{equation}
If there exists an $\alpha \in \Kinf$ such that for
$D=\diag(\id + \alpha)$ the gain operator $\Gamma_\mu$ satisfies the
strong small gain condition
\[ D \circ \Gamma_\mu(s) \not \geq s \] then the interconnected system
is ISS.
\end{corollary}

\begin{proof}
 After permutation $\ol\Gamma$ is of the
 form~\eqref{eq:reducedform}. For each of the diagonal blocks
 Corollary~\ref{sumcor} is applicable and the result follows from
 Proposition~\ref{prop:reducible}.
\end{proof}

\begin{corollary}[Maximization of gains]
 \label{cor:maximization_of_gains}
 Consider the interconnected system $\Sigma$ given by \eqref{eq:3},
 \eqref{eq:4} where each of the subsystems $\Sigma_i$ has an ISS
 Lyapunov function $V_i$ and the corresponding gain matrix is given
 by \eqref{Gammadef}. Assume that the gains enter the ISS condition
 via maximization, that is,
\begin{equation}\label{eq:max-of-gains}
  \ol{\Gamma}_{\mu,i}(V_1(x_1),\ldots, V_n(x_n),\norm{u}) =
  \max\left\{\gamma_{i1}(V_1(x_1)) ,\ldots, \gamma_{in}(V_n(x_n)) , \gamma_{iu}(\norm{u}) \right\} \,.
\end{equation}
If the gain operator $\Gamma_\mu$ satisfies the small gain condition
\[ \Gamma_\mu(s) \not \geq s \] then the interconnected system is
ISS.
\end{corollary}

\begin{proof}
 After permutation $\ol\Gamma$ is of the
 form~\eqref{eq:reducedform}. For each of the diagonal blocks
 Corollary~\ref{maxcor} is applicable and the result follows from
 Proposition~\ref{prop:reducible}.
\end{proof}

\section{Applications of the general small gain theorem}
\label{sec:appl-gener-small}

In Section \ref{examples} we have presented several examples of
functions $\mu_i$, $\gamma_i$ and gain operators $\Gamma_{\mu}$,
$\overline{\Gamma}_{\mu}$. Here we will show how our main results
apply to these examples. Before we proceed, let us consider the
special case of homogeneous $\Gamma_{\mu}$ (of degree 1)
\cite{GaG04}. Here $\Gamma_{\mu}$ is homogeneous of degree one if
for any $s\in\Rnp$ and any $r>0$ we have
$\Gamma_\mu(rs)=r\Gamma_\mu(s)$.
\begin{proposition}[Explicit paths and Lyapunov functions for
 homogeneous gain operators]
 \label{prop:homogen}
 Let $\Sigma$ in \eqref{eq:2} be a strongly connected network of
 subsystems \eqref{eq:1} and $\Gamma_{\mu}$,
 $\overline{\Gamma}_{\mu}$ be the corresponding gain operators. Let
 $\Gamma_{\mu}$ be homogeneous and let $\ol \Gamma_{\mu}$ satisfy one
 of the conditions \eqref{eq:additive-external-gain},
 \eqref{eq:max-of-gains}, or \eqref{eq:separation}. If $\Gamma_\mu$
 satisfies the strong small gain condition
 \eqref{eq:strong-small-gain-condition} (\eqref{eq:small-gain-condition}
 in case of \eqref{eq:max-of-gains}) then the interconnection
 $\Sigma$ is ISS, moreover there exists a (nonlinear) eigenvector
 $0<s\in\R^n$ of $\Gamma_\mu$ such that $\Gamma_\mu(s)=\lambda s$
 with $\lambda<1$ and an ISS Lyapunov function for the network is
 given by
 \begin{equation}
   \label{eq:LF-for-homogen}
   V(x)=\max_i\{V_i(x_i)/s_i\}.
 \end{equation}
\end{proposition}

\begin{proof}
 First note that one of the Corollaries \ref{cor:addition_of_gains},
 \ref{cor:maximization_of_gains}, or \ref{multcor} can be applied and
 the ISS property follows immediately. By the assumptions of the
 proposition we have an irreducible monotone homogeneous operator
 $\Gamma_\mu$ on the positive orthant $\R^n_+$. By the generalized
 Perron-Frobenius Theorem \cite{GaG04} there exists a positive
 eigenvector $s\in\R^n_+$. Its eigenvalue $\lambda$ is less than one,
 otherwise we have a contradiction to the small gain condition. The
 ray defined by this vector $s$ is a corresponding $\Omega$-path and
 by Theorem~\ref{LFtheo} we obtain \eqref{eq:LF-for-homogen}.
\end{proof}

One type of homogeneous operators arises from linear operators
through multiplicative coordinate transforms. In this case we can
further specialize the assumptions of the previous result.

\begin{lemma}
 \label{lemma:essentially-linear-is-homogeneous}
 Let $\alpha\in\Kinf$ satisfy\footnote{In other words,
   $\alpha(r)=r^{c}$ for some $c>0$.}
 $\alpha(ab)=\alpha(a)\alpha(b)$ for all $a,b\geq0$. Let
 $D=\diag(\alpha)$, $G\in\R_{+}^{n\times n}$, and $\Gamma_{\mu}$ be
 given by
 $$
 \Gamma_{\mu}(s) = D^{-1}(GD(s))\,.
 $$
 Then $\Gamma_{\mu}$ is homogeneous. Moreover, $\Gamma_{\mu}\ngeq\id$ if
 and only if the spectral radius of $G$ is less than one.
\end{lemma}
\begin{proof}
 If the spectral radius of $G$ is less than one, then there exists a
 positive vector $\tilde s$ satisfying $G\tilde s<\tilde s$: Just add
 a small $\delta>0$ to every entry of $G$, so that the spectral
 radius $\rho(\tilde G)$ of $\tilde G$ is still less than
 one, due to continuity of the spectrum. Then there exists a
 Perron vector $\tilde s$ such that $G\tilde s< \tilde G
 \tilde s = \rho(\tilde G) \tilde s < \tilde s$.  Define $\hat s =
 D^{-1}(\tilde s)>0$ and observe that
 $\alpha^{-1}(ab)=\alpha^{-1}(a)\alpha^{-1}(b)$.  Then we have
 \begin{equation}
   \begin{aligned}[t]
     \Gamma_{\mu}(r\hat s) = D^{-1}(G D(r\hat s)) = D^{-1}(\alpha(r)
     GD(\hat s)) &= r D^{-1}( G\tilde s)\\ &< r D^{-1}(\tilde s) =
     r\hat s\,,
   \end{aligned}
   \label{eq:14}
 \end{equation}
 for all $r>0$. So an $\Omega$-path for $\Gamma_{\mu}$ is given by
 $\sigma(r)=r\hat s$ for $r\geq0$. Existence of an $\Omega$-path
 implies the small gain condition: The origin in $\Rnp$ is globally
 attractive with respect to the system $s^{k+1}=\Gamma_{\mu}(s^{k})$,
 as can be seen by a monotonicity argument. By \cite[Theorem
 23]{DRW-mcss} or \cite[Prop. 4.1]{R08-positivity} we have
 $\Gamma_{\mu}\ngeq\id$.

 Assuming that the spectral radius of $G$ is greater or equal to one
 there exists $\tilde s\in\Rnp$, $\tilde s\ne0$, such that $G\tilde s\geq \tilde
 s$. Defining $\hat s=D^{-1}(\tilde s)$ we have $\Gamma_{\mu}(\hat s)
 = D^{-1}(GD(\hat s)) = D^{-1}(G\tilde s) \geq D^{-1}(\tilde s) =
 \hat s$. Hence $\Gamma_{\mu}\ngeq\id$ if and only if the spectral
 radius of $G$ is less than one.

 Homogeneity of $\Gamma_{\mu}$ is obtained as in~\eqref{eq:14}.
\end{proof}

\subsection{Application to linear interconnected systems}
\label{linear-application}

Consider the interconnection \eqref{linear-systems} of linear
systems from Section \ref{section:linear}.

\begin{proposition}
Let each $\Sigma_i$ in \eqref{linear-systems} be ISS with a
quadratic ISS Lyapunov function $V_i$, so that the corresponding
operator $\Gamma_\mu$ can be taken to be as in~\eqref{eq:9}.
If the spectral radius $r(G)$  of the associated matrix
\begin{equation}
  \label{eq:12}
  G=
  \begin{pmatrix}
    \frac{2b^3_i\|\Delta_{ij}\|}{c_i(1-\epsilon){a_j}}
  \end{pmatrix}_{ij}
\end{equation}
is less than 1, then the interconnection
$$
\Sigma: \quad \dot x = (A+\Delta)x+Bu
$$
is ISS and its (nonsmooth) ISS Lyapunov function can be taken as
$$
V(x)=\max_i \frac{1}{s_i} x_i^T P_i x_i
$$
for some positive vector $s\in\R_+^n$.
\end{proposition}

\begin{proof}
We have $\Gamma_{\mu} = D^{-1}(GD(\cdot))$,
where $D=\diag(\alpha)$ for $\alpha(r)=\sqrt{r}$. Now
$\alpha$ satisfies the assumptions of
Lemma~\ref{lemma:essentially-linear-is-homogeneous}, which yields
that $\Gamma_{\mu}$ satisfies the small gain condition $\Gamma_\mu \ngeq\id$
if and only if $r(G) < 1$. If $G$ or equivalently $\Gamma_\mu$ is irreducible
then there exists by Proposition~\ref{prop:homogen} an $s>0$ such that
$\Gamma_\mu(s) < s$. By \eqref{eq:9} we see, that there exists an $r^*\in (0,\infty)$ 
such that $\overline{\Gamma}_\mu(s,r^*) < s$. Then defining $\sigma(r) = r s, \phi(r) = \sqrt{r} r^*$ we obtain
for all $r>0$ that
\begin{equation*}
 \overline{\Gamma}_\mu(\sigma(r),\phi(r)) = r \overline{\Gamma}_\mu(s,r^*) < r s = \sigma(r)\,.
\end{equation*}
Thus the conditions of Theorem~\ref{LFtheo} are satisfied and an ISS Lyapunov function can be taken as
$V(x)=\max_i \frac{1}{ s_i} x_i^T P_i x_i$.

If $G$ is reducible the previous construction has to be performed for every irreducible blocks and then the scaling
techniques of Section~\ref{sec:reducible} need to be applied.
\end{proof}

\subsection{Application to neural networks}
\label{neural-networks-application}

Consider the neural network \eqref{eq:neural network}
discussed in Section~\ref{section:neural networks}. This is a
system of coupled nonlinear equations, and we have seen that each
subsystem is ISS. Note that so far we have not imposed any
restrictions on the coefficients $t_{ij}$. Moreover the
assumptions imposed on $a_i,\,b_i,\,s_i$ are essentially milder
then in \cite{WaZ02}. However to obtain the ISS property of the
network we need to require more. The small gain condition can be
used for this purpose. It will impose restrictions on the coupling
terms $t_{ij}s(x_j)$. From Corollary \ref{sumcor} it follows:

\begin{theorem}
 Consider the Cohen-Grossberg neural network~\eqref{eq:neural
   network}. Let $\Gamma_\mu$ be given by $\gamma_{ij}$ and $\mu_i$,
 $i,j=1,\dots,n$, as calculated for the interconnection in
 Section~\ref{section:neural networks}.  Assume that $\Gamma_\mu$ satisfies the
 strong small gain condition $D\circ\Gamma_\mu\not\ge\id$ for
 $s\in\R^n_+\setminus{0}$. Then this network is ISS from
 $(J_1,\dots,J_n)^T$ to $x$.
\end{theorem}

\begin{remark}
 In \cite{WaZ02} the authors have proved that there exists a unique
 equilibrium point for the network and given constant external
 inputs. They have also proved the exponential stability of this
 equilibrium. We have considered arbitrary external inputs to the
 network and proved the ISS property for the interconnection.
\end{remark}

\section{Path construction}
\label{sec:pathconstruction}

This section explains the relation between the small gain
condition for $\Gamma_\mu$ and its mapping properties. Then we
construct a strictly increasing $\Omega$-path and prove
Theorem~\ref{path} and some extensions.  Let us first consider
some simple particular cases to explain the main ideas, as
depicted in Figure~\ref{fig:path-elements}.  In the following
subsections we then proceed to the main path construction results.

\begin{figure}[htbp]
\centering
\hfill
\includegraphics[width=0.35\textwidth]{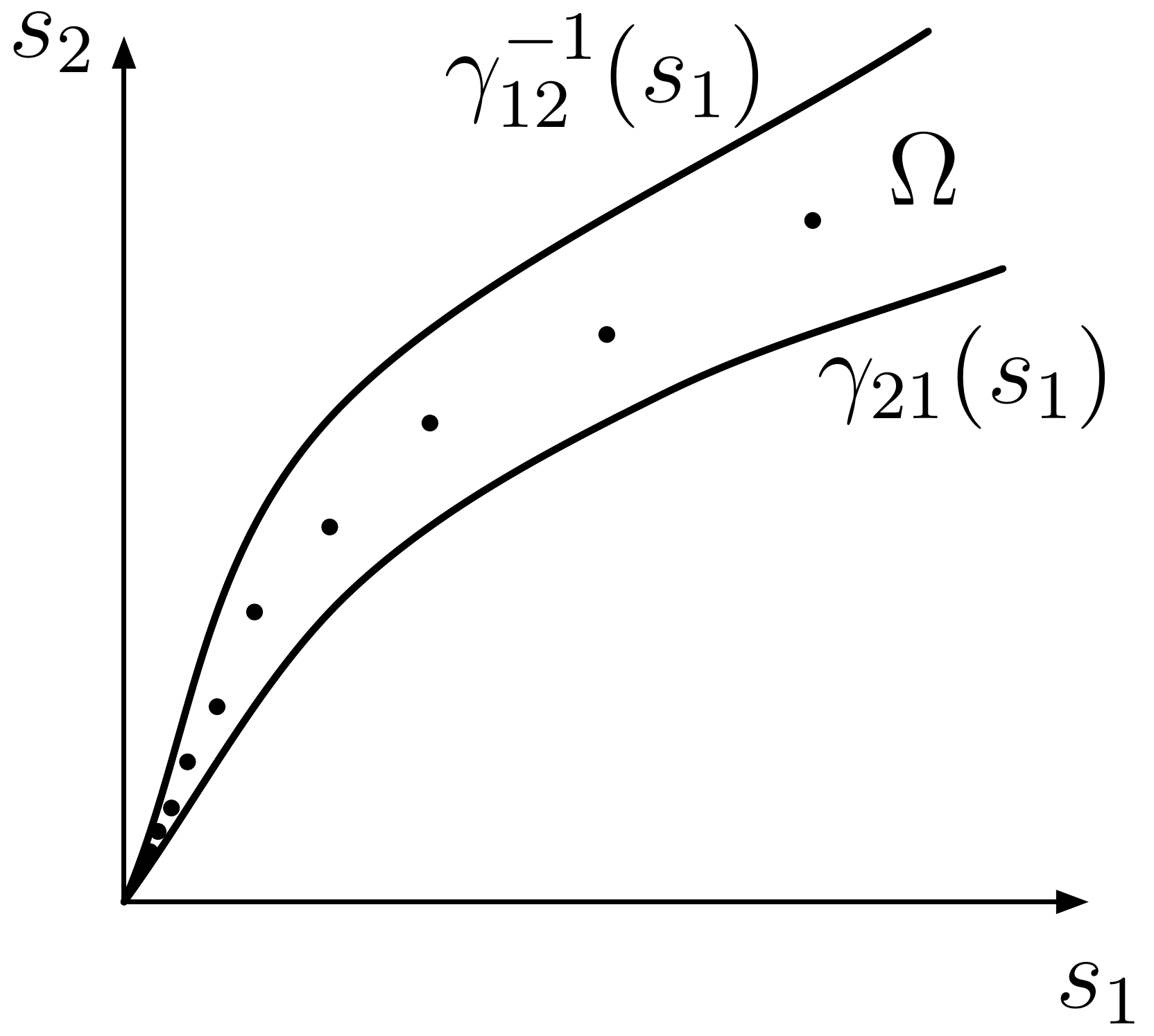}
\hfill
\includegraphics[width=0.35\textwidth]{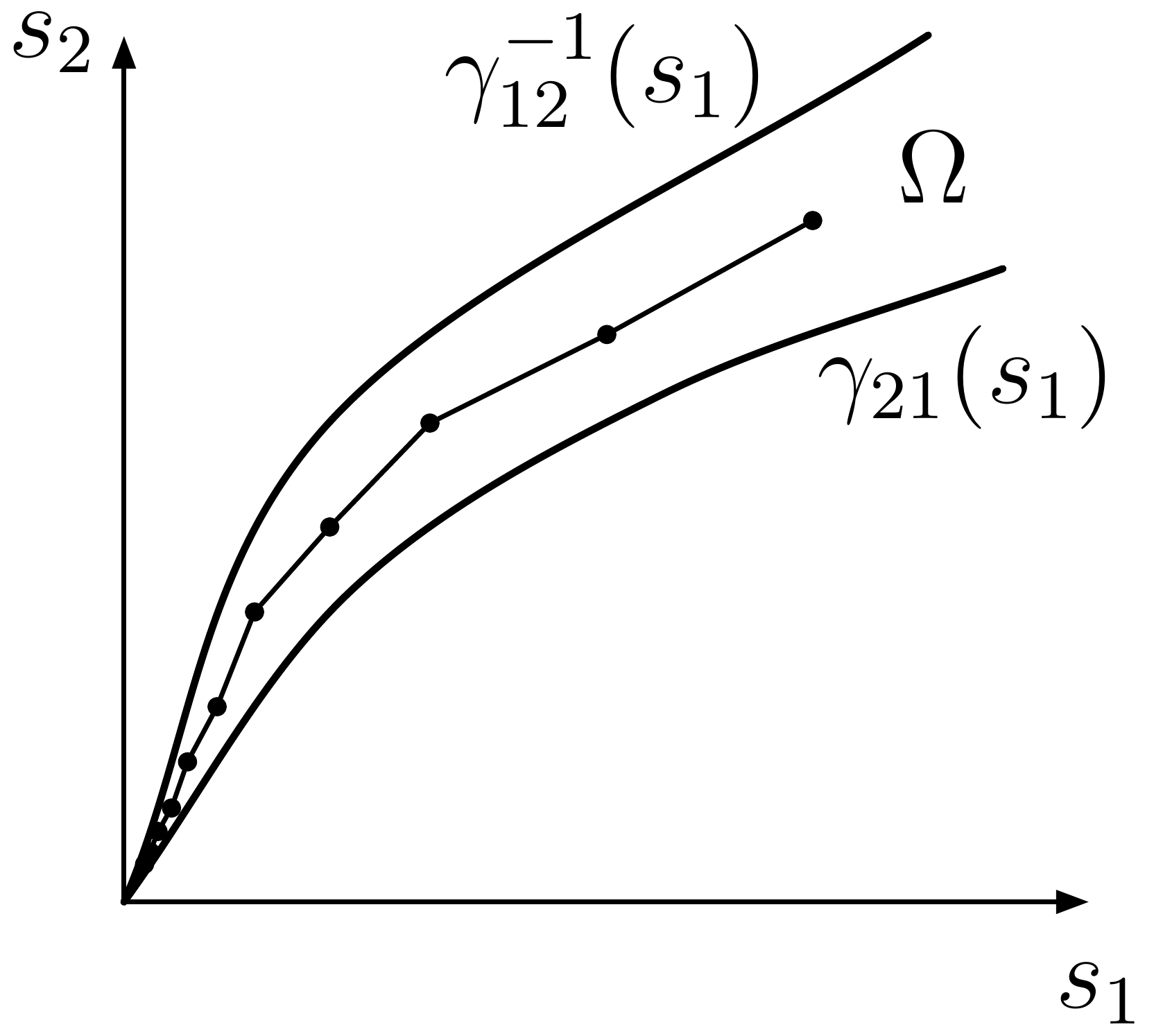}
\hfill\mbox{}
\caption{A sequence of points $\{\Gamma_{\mu}^{k}(s)\}_{k\geq0}$ for
  some $s\in \Omega(\Gamma_{\mu})$, where
  $\Gamma_{\mu}:\R^{2}_{+}\to\R^{2}_{+}$ is given by
  $\Gamma_{\mu}(s)=(\gamma_{12}(s_{2}),\gamma_{21}(s_{1}))^{T}$ and
  satisfies
  $\Gamma_{\mu}\ngeq\id$, or, equivalently, $\gamma_{21}\circ
  \gamma_{12}<\id$, and the corresponding linear interpolation, cf.\
  Lemmas~\ref{lem:easyconv}, \ref{lem:pathconvex}, and \ref{lemma:sigma-s}.}
\label{fig:path-elements}
\end{figure}

A map $T:\Rnp\to\Rnp$ is \emph{monotone} if $x\leq y$ implies
$T(x)\leq T(y)$. Clearly any matrix $\Gamma\in\Kinfnn$ together
with an aggregation $\mu \in \MAF_n^n$ induces a monotone map $\Gamma_{\mu}$.

\begin{lemma}
\label{lem:easyconv} Let $\Gamma\in \Knn$ and $\mu \in \MAF_n^n$,
such that $\Gamma_\mu$ satisfies \eqref{eq:small-gain-condition}.
If $s\in \Omega(\Gamma_\mu)$, then $\lim_{k\to \infty}
\Gamma_\mu^k(s) = 0$.
\end{lemma}

\begin{proof}
If $s\in \Omega$, then $\Gamma_\mu(s) < s$ and by monotonicity
$\Gamma_\mu^2(s) \leq \Gamma_\mu(s)$. By induction
$\Gamma_\mu^k(s)$ is a monotonically decreasing sequence bounded
from below by $0$. Thus $\lim_{k\to \infty} \Gamma_\mu^k(s) =:
s^*$ exists and by continuity we have $\Gamma_\mu(s^*) = s^*$. By
the small gain condition it follows $s^*=0$.
\end{proof}

\begin{lemma}
\label{lem:pathconvex} Assume that $\Gamma\in \Knn$ has no zero
rows and let $\mu \in \MAF_n^n$.  If $0 < s\in
\Omega(\Gamma_\mu)$, then
\begin{enumerate}
\item\label{item:1} $0 < \Gamma_\mu(s) \in \Omega$
\item\label{item:2} for all $\lambda \in  [0,1]$ the convex
combination $s_\lambda
 := \lambda s + (1-\lambda) \Gamma_\mu(s) \in \Omega$.
\end{enumerate}
\end{lemma}

\begin{proof}
 \ref{item:1}~By assumption $\Gamma_\mu(s) < s$ and so by the
 monotonicity assumption (M2) we have $\Gamma_\mu(\Gamma_\mu(s)) <
 \Gamma_\mu(s)$.  Furthermore, as $s>0$ and the matrix $\Gamma$ has
 no zeros rows, we have that $\Gamma_\mu(s)>0$ by assumption (M1).

\ref{item:2}~As $\Gamma_\mu(s) < s$ it follows for all $\lambda
\in (0,1)$ that $\Gamma_\mu(s) < s_\lambda < s$.  Hence by
monotonicity and using \ref{item:1}
\[ 0< \Gamma_\mu(\Gamma_\mu(s)) < \Gamma_\mu(s_\lambda) <
\Gamma_\mu(s) < s_\lambda \,.\] This implies $s_\lambda \in
\Omega$ as desired.
\end{proof}

\begin{lemma}
\label{lemma:sigma-s} Assume that $\Gamma\in \Knn$ has no zero
rows and let $\mu \in \MAF_n^n$ be such that $\Gamma_\mu$
satisfies the small gain condition
\eqref{eq:small-gain-condition}. Let $s\in\Omega(\Gamma_{\mu})$.
Then there exists a path in $\Omega\cup\{0\}$ connecting the
origin and $s$.
\end{lemma}
\begin{proof}
By Lemma~\ref{lem:pathconvex}, the line segment $\{ \lambda
\Gamma_\mu(s) + (1-\lambda)s \} \subset \Omega$. By induction all
the line segments $\{ \lambda \Gamma_\mu^{k+1}(s) +
(1-\lambda)\Gamma_\mu^k(s) \} \subset \Omega$ for $k\geq 1$.  Using
Lemma~\ref{lem:easyconv} we see that $\Gamma_\mu^k(s)\to0$ as
$k\to\infty$.  This constructs an $\Omega$-path with respect to
$\Gamma_\mu$ from $0$ to $s$.
\end{proof}

The following result applies to $\Gamma$ whose entries are
bounded, i.e., in $(\K\setminus\Kinf)\cup\{0\}$.

\begin{proposition}
\label{prop:sigma-for-bounded-Gamma} Assume that $\Gamma\in \Knn$
has no zero rows and let $\mu \in \MAF_n^n$ be such that
$\Gamma_\mu$ satisfies the small gain condition
\eqref{eq:small-gain-condition}. Assume furthermore that
$\Gamma_\mu$ is bounded, then there exists an $\Omega$-path with
respect to $\Gamma_\mu$.
\end{proposition}

\begin{proof}
By assumption the set $\Gamma_\mu(\Rnp)$ is bounded, so pick $s >
\sup \Gamma_\mu(\Rnp)$. Then clearly, $\Gamma_\mu(s) < s$ and so
$s \in \Omega$. By the same argument $\eta s \in \Omega$ for all
$\eta \in [1,\infty)$. Thus a path in $\Omega$ through the point
$s$ exists, if we find a path from $s$ to $0$ contained in
$\Omega$. The remainder of the result is given by
Lemma~\ref{lemma:sigma-s}.
\end{proof}

The difficulty now arises if $\Gamma_\mu$ happens to be unbounded,
i.e., $\Gamma$ contains entries of class $\Kinf$. In the unbounded
case the simple construction above is not possible. In the
following we will first consider the case that all nonzero entries
of $\Gamma$ are of class $\Kinf$. Beforehand we introduce a few
technical lemmas.

\subsection{Technical lemmas}
\label{sec:technical-lemmas} Throughout this subsection
$T:\Rnp\to\Rnp$ denotes a continuous, monotone map, i.e., $T$
satisfies $T(v)\leq T(w)$ whenever $v\leq w$. We start with a few
observations.

\begin{lemma}
\label{lemma:monotone:id+rho-inverse-is-id-rhotilde} Let
$\rho\in\Kinf$. Then there exists a $\tilde\rho\in\Kinf$ such that
$(\id+\rho)^{-1} = \id -\tilde\rho$.
\end{lemma}
\begin{proof}
Just define $\tilde\rho=\rho\circ(\id+\rho)^{-1}$. Then
$(\id-\tilde\rho)\circ(\id+\rho)=(\id+\rho) -
\tilde\rho\circ(\id+\rho) =
\id+\rho-\rho\circ(\id+\rho)^{-1}\circ(\id+\rho)=\id+\rho-\rho=\id$,
which proves the lemma.
\end{proof}

\begin{lemma}
\label{lemma:D-factorization}
\begin{enumerate}
\item Let $D=\diag(\rho)$ for some $\rho\in\Kinf$ such that
 $\rho>\id$. Then for any $k\geq0$ there exist
 $\rho^{(k)}_1,\rho^{(k)}_2\in\Kinf$ satisfying $\rho^{(k)}_i>\id$,
 such that for $D^{(k)}_i=\diag(\rho^{(k)}_i)$, $i=1,2$,
 \[
 D=D^{(k)}_1 \circ D^{(k)}_2.
 \]
 Moreover, $D_2^{(k)}$, $k\geq0$, can be chosen such that for all
 $0< s\in\Rnp$ we have $$D_2^{(k)}(s)< D_2^{(k+1)}(s).$$
\item Let $D=\diag(\id+\alpha)$ for some $\alpha\in\Kinf$. Then
 there exist $\alpha_1,\alpha_2\in\Kinf$, such that for
 $D_i=\diag(\id+\alpha_i)$, $i=1,2$,
 \[
 D=D_1 \circ D_2.
 \]
\end{enumerate}
\end{lemma}
For maps $T:\Rnp\to \Rnp$ define the \emph{decay set}
$$
\Psi(T) := \{ s\in\Rnp:\, T(s)\leq s\}\,,
$$
where we again omit the reference to $T$ if this is clear from the
context.

\begin{lemma}
\label{lem:properties-Gamma-Psi-etc} Let $T:\Rnp\to\Rnp$ be
monotone and $D=\diag(\rho)$ for some $\rho\in\Kinf,\rho>\id$.
Then
\begin{enumerate}
\item $T^{k+1}(\Psi)\subset T^k(\Psi)$ for all $k\geq0$; \item
$\Psi(D\circ T)\cap\{s\in\Rnp: s>0\} \subset\Omega(T)$, if $T$
satisfies
 $T(v)<T(w)$ whenever $v<w$; the same is true for $D\circ T$
 replaced by $T\circ D$;
\end{enumerate}
\end{lemma}

The proofs of the lemmas are simple and thus omitted for reasons
of space. Nevertheless they can be found in \cite[p.10,
p.29]{07-R-DISS}.

We will need  the following connectedness property in the sequel.

\begin{proposition}
\label{prop:psi-pathwise-connected-AND-finite-length-Omega-path}
Let $\Gamma\in \Knn$ and $\mu \in \MAF_n^n$ be such that
$\Gamma_\mu$ satisfies the small gain condition
\eqref{eq:small-gain-condition}. Then $\Psi$ is nonempty and
pathwise connected. Moreover, if $\Gamma_\mu$ satisfies
$\Gamma_\mu(v)<\Gamma_\mu(w)$ whenever $v<w$, then for any
$s\in\Omega(\Gamma_\mu)$ there exists a strictly increasing
$\Omega$-path connecting $0$ and $s$.
\end{proposition}

\begin{proof}
Note that always $0\in\Psi$, hence $\Psi$ cannot be empty. Along the
lines the proof of Lemma~\ref{lemma:sigma-s}
it follows that each point in $\Psi$ is pathwise connected to the
origin.
\end{proof}

Another crucial step, which is of topological nature, regards
preimages of points in the decay set $\Psi$. In general it is not
guaranteed, that for $s\in\Rnp$ with $T(s)\in\Psi$, we also have
$s\in\Psi$. The set of points in $\Psi$ for which preimages of
arbitrary order are also in $\Psi$ is the set
$$
\Psi_\infty(T) := \bigcap_{k=0}^\infty T^k(\Psi).
$$
Of course, this set might be empty or bounded. We will use it to
construct $\Omega$-paths for operators $\Gamma_\mu$ satisfying the
small gain condition.

\begin{proposition}[{\cite[Prop.~5.4]{R08-positivity}}]
\label{prop:psi-infty-unbounded} Let $T:\Rnp\to\Rnp$ be monotone
and continuous and satisfy $T(s)\ngeq s$ for all $s\ne0$.  Assume
that $T$ satisfies the property
\begin{equation}
 \label{eq:unboundedness-assumption}
 \|s_k\|\to\infty \implies
 \|T(s_k)\|\longrightarrow\infty
\end{equation}
as $k\to\infty$ for any sequence $\{s_k\}_{k\in\N}\subset\Rnp$.

Then $\Psi_\infty(T)\subset\Psi(T)$, $\Psi_\infty(T)\cap
S_r\ne\emptyset$ for all $r\geq0$, and $\Psi_\infty(T)$ is
unbounded.
\end{proposition}

\begin{figure}[htb]
\centering
\includegraphics[viewport=0 0 2500 1698,width=.618\textwidth]{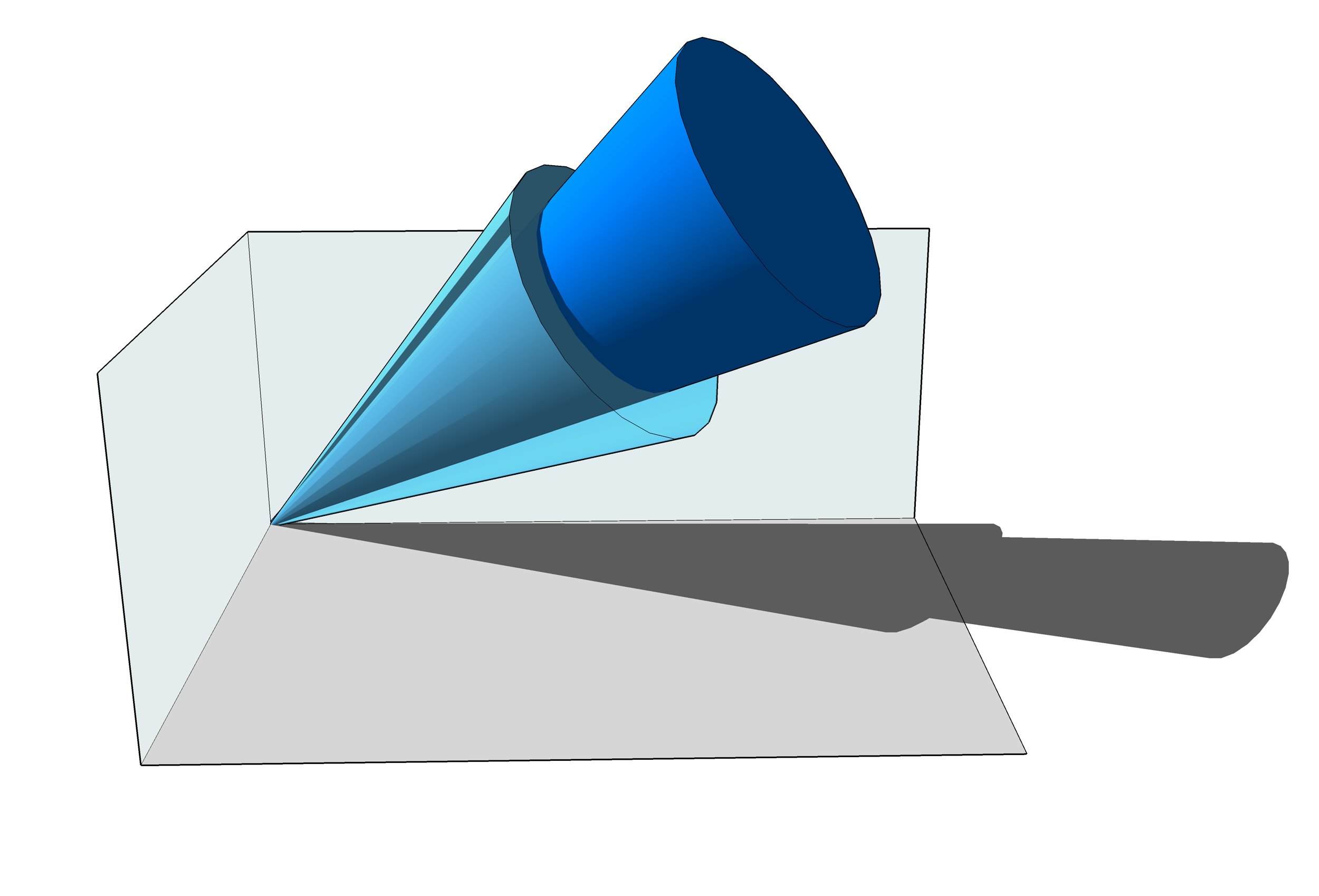}
\caption{A sketch of the set $\Psi_\infty\subset\Psi\subset\Rnp$
in
 Proposition~\ref{prop:psi-infty-unbounded}.}
\label{fig:psi-infty}
\end{figure}

A result based on the topological fixed point theorem due to
Knaster, Kuratowski, and Mazurkiewicz allows to relate $\Omega$
and the small gain condition. It is essential for the proof of
Proposition~\ref{prop:psi-infty-unbounded}.

\begin{proposition}
\label{prop:KKM-result} Let $T:\Rnp\to\Rnp$ be monotone and
continuous. If $T(s)\ngeq s$ for all $s\in\Rnp$ then the set
$\Omega\cap S_r$ is nonempty for all $r>0$.
\end{proposition}

In particular, $s\in\Omega\cap S_r$ for $r>0$ implies $s>0$. The
proof for this result can be found in \cite[Prop. 1.5.3,
p.26]{07-R-DISS} or in a slightly different form in
\cite{DRW-mcss}.

\subsection{Paths for $\Kinf\cup\{0\}$ gain matrices}
\label{sec:Kinf-case}

In this subsection we consider matrices $\Gamma\in\Kinfnn$, i.e.,
all nonzero entries of $\Gamma$ are assumed to be unbounded
functions.

In this setting we assume and utilize that the graph associated to
$\Gamma$ is strongly connected, i.e., $\Gamma$ is irreducible. So that
if we consider powers $\Gamma_\mu^k(x)$, for each components $i$ and
$j$ there exists a $k=k(i,j)$ such that $t\mapsto \Gamma_\mu^k(t\cdot
e_j)_i$ is an unbounded function.

\begin{theorem}
\label{thm:monotone-path-irreducible-case} Let $\Gamma\in\Kinfnn$
be irreducible, $\mu\in\MAF_n^n$, and assume $\Gamma_\mu\ngeq\id$.
Then there exists a strictly increasing path $\sigma\in\Kinf^n$
satisfying
$$
\Gamma_\mu(\sigma(r))<\sigma(r),\quad \forall r>0.
$$
\end{theorem}

The main technical difficulty in the proof is to construct the
path in the unbounded direction, the other case has already been
dealt with in
Proposition~\ref{prop:psi-pathwise-connected-AND-finite-length-Omega-path}.

The proof comprises the following steps: First due to
\cite[Prop.~5.8]{R08-positivity} we may choose a $\Kinf$ function
$\phi>\id$ so that for $D=\diag(\phi)$ we have $\Gamma_\mu\circ D
\ngeq\id$. Then we construct a monotone (but not necessarily
strictly monotone) sequence $\{s^k\}_{k\geq0}$ in
$\Psi(\Gamma_\mu\circ D)$, satisfying
$s^{k}=\Gamma_{\mu}(D(s^{k+1})) \lneqq s^{k+1}$, so that each
component sequence is unbounded. At this point a linear
interpolation of the sequence points may not yield a
\emph{strictly increasing} path. So finally we use the ``extra
space'' provided by $D$ in the set $\Omega(\Gamma_\mu)\supset
\Omega(\Gamma_\mu\circ D)$ to obtain a strictly increasing
sequence $\{\tilde s^k\}_{k\geq0}$ in $\Omega(\Gamma_\mu)$ which
we can linearly interpolate to obtain the desired $\Omega$-path.

\begin{proof}
Since $\Gamma$ is irreducible, it has no zero rows and hence
$\Gamma_\mu$ satisfies $\Gamma_\mu(v)<\Gamma_\mu(w)$ whenever
$v<w$. By \cite[Prop.~5.8]{R08-positivity} there exists a
$\phi>\id$ so that for $D=\diag(\phi)$ we have $\Gamma_\mu\circ D
\ngeq\id$.  Now we construct a nondecreasing sequence $\{s^k\}$ in
$\Psi(\Gamma_\mu\circ D)$:

Let $T:=\Gamma_\mu\circ D$. Then $T$ and by induction also all
powers $T^l$, $l\geq1$,
satisfy~\eqref{eq:unboundedness-assumption}.

By Proposition~\ref{prop:psi-infty-unbounded} the set $\Psi_\infty(T)$
is unbounded, so we may pick an $0 \neq s^0 \in \Psi_\infty(T)$.  We
can actually choose $s^0>0$, since the sequence $\{s^k\}$ we are going
to construct will be unbounded in every component, at which point we
may replace $s^{0}$ by some $s^{k}>0$ for $k$ large enough.

Due to irreducibility of $\Gamma$ (and Remark~\ref{genass}) the
following property holds: For any pair $1\leq i,j\leq n$ there
exists an $l\geq1$ such that
\begin{equation}
 r\mapsto (\Gamma_\mu^l (r e_j))_i
 \label{eq:13}
\end{equation}
is an unbounded and increasing function, where $e_j$ is the $j$-th
unit vector. By monotonicity the same holds when $T$ is considered
instead of $\Gamma_{\mu}$.  Now define a sequence
$\{s^k\}_{k\geq0}$ by choosing
$$
s^{k+1} \in T^{-1}(s^k)\cap\Psi_\infty(T)
$$
for $k\geq0$. This is possible, since by definition
$\Psi_\infty(T)$ is backward invariant under $T$.

This sequence $\{s^k\}$ satisfies $s^k\lneqq s^{k+1}$ by
definition. We claim that it is unbounded, and also unbounded in
every component: To this end assume first that it is bounded. Then
by monotonicity there exists a limit $s^*=\lim_{k\to\infty} s^k$.
By continuity of $T$ and since $s^k=T(s^{k+1})$ we have
$$
s^*=\lim_{k\to\infty} s^k = \lim_{k\to\infty}
T(s^{k+1})=T\left(\lim_{k\to\infty} s^{k+1}\right) = T(s^*)
$$
contradicting $T(s)\ngeq s$ for all $s\ne0$. Hence the sequence
$\{s^k\}$ must be unbounded.

Let $j$ be an index such that $\{s^k_j\}_{k\in\N}$ is unbounded,
let $i\in\{1,\ldots,n\}$ be arbitrary and choose $l$ such that the
function in \eqref{eq:13} is unbounded for $i,j,l$. Choose real
numbers $r_k \to \infty$ such that $r_k e_j \leq s^k$ for all
$k\in \N$. Then we have
\[ (T^l (r_k e_j))_i \leq (T^l (s^k))_i = s^{k-l}_i \,.\] As the term
on the left goes to $\infty$ for $k\to \infty$, so does $s^k_i$.
Hence $\{s^k\}$ is unbounded in every component.

Now by Lemma~\ref{lem:properties-Gamma-Psi-etc}(ii) the sequence
$\{s^k\}$ is contained in $\Omega(\Gamma_\mu)$, but it may not be
strictly increasing, as we only know $s^k\lneqq s^{k+1}$ for all
$k\geq0$. We define a strictly increasing sequence $\{\tilde
s^k\}$ as follows: By Lemma~\ref{lemma:D-factorization} for any
$k\geq0$ we may factorize $D=D_1^{(k)}\circ D_2^{(k)}$ so that 
$D_1^{(k)},D^{(k)}_2 > \id$ and
$D_2^{(k)}(s)< D_2^{(k+1)}(s)$ for all $k\geq0$ and all
$s>0$. Using this factorization we define
$$
\tilde s^k := D_2^{(k)}(s^k)
$$
for all $k\geq0$. By the definition of $D_2^{(k)}$, this sequence
is clearly strictly increasing and inherits from $\{s^k\}$ the
unboundedness in all components.

We claim that $\{\tilde s^k\}\subset \Omega(\Gamma_\mu)$: This
follows from
$$
\tilde s^k > s^k \geq \Gamma_\mu \circ D (s^k) = \Gamma_\mu \circ
D_1^{(k)} \circ D_2^{(k)} (s^k) = \Gamma_\mu \circ D_1^{(k)}
(\tilde s^k) > \Gamma_\mu(\tilde s^k).
$$

Now we prove that for $\lambda\in(0,1)$ we have $(1-\lambda)\tilde
s^k + \lambda\tilde s^{k+1} \in \Omega(\Gamma_\mu)$.  Clearly
$$
\tilde s^k < (1-\lambda)\tilde s^k + \lambda\tilde s^{k+1} <
\tilde s^{k+1}
$$
and application of the strictly increasing operator $\Gamma_\mu$
yields
\begin{multline*}
 \Gamma_\mu((1-\lambda)\tilde s^k + \lambda\tilde s^{k+1}) <
 \Gamma_\mu(\tilde s^{k+1})\\
 = \Gamma_\mu \circ D_2^{(k+1)} (s^{k+1}) < \Gamma_\mu \circ
 D_1^{(k+1)} \circ D_2^{(k+1)} (s^{k+1})\\
 = s^k < \tilde s^k < (1-\lambda)\tilde s^k + \lambda\tilde s^{k+1}.
\end{multline*}
Hence $(1-\lambda)\tilde s^k + \lambda\tilde s^{k+1}\in
\Omega(\Gamma_\mu)$.

Now we may define $\sigma$ as a parametrization of the linear
interpolation of the points $\{\tilde s^k\}_{k\geq0}$ in the
unbounded direction and utilize the construction from
Lemma~\ref{lemma:sigma-s} for the other direction. Clearly this
function $\sigma$ 
is an $\Omega$-path as it
has component functions of class $\Kinf$ and is
piecewise linear on every compact interval contained in
$(0,\infty)$.
\end{proof}

It is possible to consider the reducible case in a similar
fashion. The argument is essentially an induction over the number
of irreducible and zero blocks on the diagonal of the reducible
operator. We cite the following result from \cite[Theorem
5.10]{R08-positivity}. However, for the construction of an ISS
Lyapunov function in the case of reducible $\Gamma$, we take a
different route as described in Section~\ref{sec:reducible}, thus
avoiding the use of assumption (M4).

\begin{theorem}
\label{thm:monotone-path-reducible-case} Let $\Gamma\in\Kinfnn$ be
reducible, $\mu\in\MAF_n^n$ satisfying (M4), $D=\diag(\id+\alpha)$
for some $\rho\in\Kinf$, and assume $\Gamma_\mu\circ D \ngeq\id$.
Then there exists a monotone and continuous operator $\tilde
D:\Rnp\to\Rnp$ and a strictly increasing path $\sigma:\Rp\to\Rnp$
whose component functions are all unbounded, such that
$\Gamma_\mu\circ \tilde D(\sigma)<\sigma$.
\end{theorem}

\subsection{General  $\mathbf\Gamma_{\mathbf\mu}$}
\label{sec:general-unbounded-Gamma-mu}

In the preceding subsections we have seen that it is possible to
construct $\Omega$-paths for matrices $\Gamma$ whose nonzero
entries are either all bounded, or all unbounded. It remains to
consider the case that the nonzero entries of $\Gamma$ are partly
of class $\Kinf$ and partly of class $\K\setminus\Kinf$. We can
state the following result.

\begin{proposition}
\label{prop:partly-bounded-Gamma}
Let $\Gamma\in\Knn$ and let $\mu\in\MAF_{n}^{n}$ satisfy
(M4). Assume $\Gamma_{\mu}$ satisfies
\eqref{eq:strong-small-gain-condition}. Then there exists an
$\Omega$-path for $\Gamma_{\mu}$.
\end{proposition}

\begin{trivlist}
\item \emph{Proof.} Write
 $$\Gamma=\Gamma_U+\Gamma_B$$
 with $\Gamma_U\in\Kinfnn$,
 $\Gamma_B\in(\K\setminus\Kinf\cup\{0\})^{n\times n}$.
 Clearly we have $(\Gamma_U)_\mu\leq\Gamma_\mu$
 and $(\Gamma_B)_\mu\leq\Gamma_\mu$ and hence both maps satisfy
 $$(\Gamma_\bullet)_\mu\ngeq\id,$$
 where $\bullet$ serves as a placeholder for the subscripts
 $U$ and $B$.
\item The map $(\Gamma_B)_\mu$ is bounded. Hence $s^* := \sup
 (\Gamma_B)_\mu(\Rnp)$ is a finite vector.
\item    By Theorem~\ref{thm:monotone-path-reducible-case}.
 for $(\Gamma_U)_\mu$ there exists a $\Kinf$ function
 $\tilde\rho$ and a $\Kinf$-path
 $\sigma_U$ so that for the diagonal operator $\tilde
 D=\diag(\id+\tilde\rho)$ we have
 $$ ((\Gamma_U)_\mu \circ \tilde D) (\sigma_U(r)) <
 \sigma_U(r),\quad \text{for all } r>0\,.
 $$
\item Similarly, by
 Proposition~\ref{prop:sigma-for-bounded-Gamma}, there exists a
 $\Kinf$-path $\sigma_B$ such that
 $(\Gamma_B)_\mu(\sigma_B(r))<\sigma_B(r)$ for all $r>0$. In
 fact, and this is the key to this proof, it is possible to
 choose $\sigma_B$ in the region where $\sigma_B(r)>s^*$ to grow
 arbitrarily slowly:
 For any $\alpha,\beta\in\Kinf$ we can find a $\kappa\in\Kinf$, such
 that
 $$
 (\alpha\circ \kappa) (r)<\beta(r),\quad r>0,
 $$
 e.g., by choosing $\kappa\in\Kinf$ satisfying
 $\kappa(r)<(\alpha^{-1}\circ \beta)(r)$. This is always possible.
 Denote $\bar D=\diag(\tilde\rho)$, (so
 that $\tilde D = \id +\bar D$) and choose $r^*$, such that $\bar D(\sigma_U(r^*)) > s^*$.
 Then after reparametrization we may assume
 that
 \begin{gather*}
   \sigma_B(r)<\bar D(\sigma_U(r))\quad \text{and}\quad
   \sigma_B(r)>s^*
 \end{gather*}
 for all $r\geq r^*$.
Using Lemma~\ref{lemma:sigma-s}, we let
 $\sigma_L:[0,r^*]\to\Rnp$ be a finite-length path satisfying
 \begin{gather*}
   \Gamma_\mu(\sigma_L(r))<\sigma_L(r),\quad \forall r\in(0,r^*],
   \\
   \sigma_L\text{ is strictly increasing}\\
   \sigma_L(0)=0\text{ and } \sigma_L(r^*)=\sigma_B(r^*)+\sigma_U(r^*).
 \end{gather*}
 Now define $\sigma$ by
 $$
 \sigma(r)=
 \begin{cases}
   \sigma_B(r) + \sigma_U(r) & \text{ if } r>r^*\\
   \sigma_L(r) & \text{ if } r< r^*.
 \end{cases}
 $$
\item It remains to check that $\sigma$ satisfies
 $\Gamma_\mu(\sigma(r))<\sigma(r)$ for $r\geq r^*$. Indeed, for
 $r\geq r^*$ we have
 \begin{align*}
   \sigma(r)&=\sigma_U(r)+\sigma_B(r) > ((\Gamma_U)_\mu\circ \tilde D)(\sigma_U(r)) + s^*\\
   &> (\Gamma_U)_\mu(\sigma_U(r) + \sigma_B(r)) +
   (\Gamma_B)_\mu(\sigma_U(r)+\sigma_B(r))\\
   &\geq \Gamma_\mu(\sigma_U(r)+\sigma_B(r)),
 \end{align*}
 where the last inequality is due to (M4). This completes the
 proof. \endproof
\end{trivlist}

\subsection{Special case:  Maximization}
\label{sec:special-case:mu=max}

The case when the aggregation is the maximum, i.e., $\mu=\max$, is
indeed a special case, since not only the small gain condition can
be formulated in simpler manner, but also the path construction
can be achieved without the need of the diagonal operator $D$ as
before.

A \emph{cycle} in a matrix $\Gamma$ is finite sequence of nonzero
entries of $\Gamma$ of the form
$$
(\gamma_{i_1,i_2},\gamma_{i_2,i_3},\ldots,\gamma_{i_K,i_1}).
$$
A cycle is called \emph{subordinated} if $i_1>
\max\{i_2,\ldots,i_K\}$, and it is called a \emph{contraction}, if
$$
\gamma_{i_1,i_2}\circ
\gamma_{i_2,i_3}\circ\ldots\circ\gamma_{i_K,i_1} < \id.
$$
It is an easy exercise to show that when all subordinated cycles
are contractions then already all cycles are contractions.

\begin{theorem}
\label{thm:mu=max} Let $\mu=\max$ and $\Gamma\in\Knn$. If all
subordinated cycles of $\Gamma$ are contractions, then there
exists an $\Omega$-path with respect to $\Gamma_\mu$.
\end{theorem}

The proof is composed of the following steps. The first step is to
show that the cycle condition (all cycles being contractions) is
equivalent to $\Gamma_\mu\ngeq\id$. Note that $\mu=\max$
automatically satisfies (M4), but (M4) is actually not needed for
the proof. Then the path-construction can essentially be done as
before, replacing sums by maximization, and one can even 
avoid
the use of $D=\diag(\id+\rho)$. Cf.\ also\cite{R08-positivity}.

\subsection{Proof of Theorem~\ref{path}}
\label{sec:general-case-K}

We now come to the easiest part of this section, which is to
combine all the preceding results to one general theorem for
matrices with entries of class $\K$, namely Theorem~\ref{path}.

\noindent\emph{Proof of Theorem~\ref{path}.}
\begin{enumerate}
\item In the linear case we can identify $\Gamma_\mu$ with a real
matrix with nonnegative entries. Then there exists a positive
vector $v>0$ so that $\Gamma_\mu v<v$ if the spectral radius
$\rho(\Gamma_\mu)<1$, cf.\ \cite{BeP79} or \cite[Lemma 2.0.1,
p.33]{07-R-DISS}. For $r>0$ this gives $\Gamma_\mu rv<rv$, i.e., a
$\Kinf$-path is given by $\sigma(r)=rv$. \item This is
Theorem~\ref{thm:monotone-path-irreducible-case}. \item This is
Theorem~\ref{thm:mu=max}. \item This is
Proposition~\ref{prop:sigma-for-bounded-Gamma}.
\endproof
\end{enumerate}

\section{Remarks for the case of three subsystems}
\label{sec:remarks-case-three}

Recall that a construction of an $\Omega$-path $\sigma$ for the
case of two subsystem was given in \cite{JMW96}. We have seen that
in a general case of $n\in\N$ subsystems the construction involves
more theory and topological properties of $\Gamma_\mu$ that follow
from the small gain condition. However in case of three subsystems
$\sigma$ can be found by rather simple considerations. Here we
provide this illustrative construction.  Let us consider the
special case $\Gamma\in(\Kinf\cup\{0\})^{3\times 3}$,
$\mu_{i}(s)=s_1+s_2+s_3$, $i=1,2,3$, and for simplicity assume
that $\gamma_{ij}\in\Kinf$ for all $i\ne j$, so that
\begin{equation}\label{eq:three-systems}
\Gamma= \left[\begin{array}{ccc}
0 & \gamma_{12} & \gamma_{13}\\
\gamma_{21} & 0 & \gamma_{23}\\
\gamma_{31} & \gamma_{32} & 0\\
\end{array}\right],\quad
\Gamma_\mu(s)= \left(\begin{array}{c}
\gamma_{12}(s_2) + \gamma_{13}(s_3)\\
\gamma_{21}(s_1) + \gamma_{23}(s_3)\\
\gamma_{31}(s_1) + \gamma_{32}(s_2)\\
\end{array}\right)\not\ge
\left(\begin{array}{c} s_1\\s_2\\s_3
\end{array}\right)
\end{equation}
Fix $s_{1}\geq0$, then it follows that there is exactly one $s_2$
satisfying
\begin{equation}\label{stern}
\gamma_{13}^{-1}(s_1-\gamma_{12}(s_2))=\gamma_{23}^{-1}(s_2-\gamma_{21}(s_1))\,,
\end{equation}
since, for a fixed $s_1$ the left side of \eqref{stern} is a
strictly decreasing function of $s_2$ while the right side of
\eqref{stern} is strictly increasing one.  The small gain
condition \eqref{eq:three-systems} in particular assures that
$\gamma_{12}^{-1}(\gamma_{21}^{-1}(r))>r$ for any $r>0$.
Let $s_2^\ast$ be the solution of $s_1-\gamma_{12}(s_2)=0$ and
$s_2^{\ast\ast}$ be the solution of $s_2-\gamma_{21}(s_1)=0$ then
$$s_2^{\ast}=\gamma_{12}^{-1}(s_1)=\gamma_{12}^{-1}(\gamma_{21}^{-1}(s_2^{\ast\ast}))>s_2^{\ast\ast}.$$
Hence the root of the left side of \eqref{stern} is greater than the
root of the right side of \eqref{stern}. This proves that for any
$s_1$ there is always exactly one $s_2$ satisfying \eqref{stern}, see Figure~\ref{fig:3systems_intersection}.
\begin{figure}[h]
\centering
      {\includegraphics[width=.5\columnwidth]{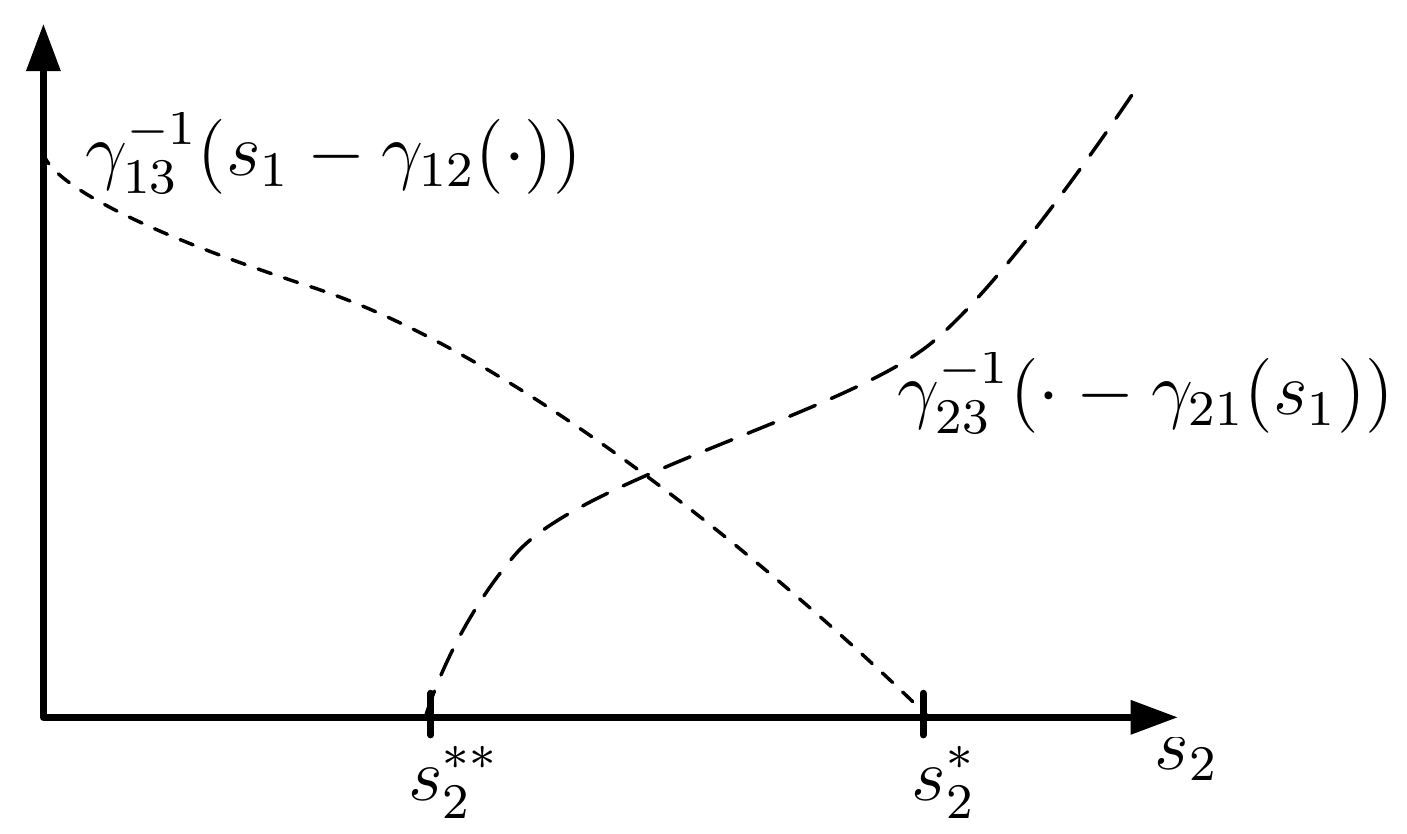}}
\caption{Visualization of \eqref{stern}.} \label{fig:3systems_intersection}
\end{figure}

By the continuity and monotonicity of
$\gamma_{12},\gamma_{21},\gamma_{13},\gamma_{23}$ it follows that
$s_2$ depends continuously on $s_1$ and is strictly increasing
with $s_1$. We can define $\sigma_1(r)=r$ for $r\geq0$ and
$\sigma_2(r)$ to be the unique $s_{2}$ solving \eqref{stern} for
$s_1=r$.

Denote $h(r)=\gamma_{31}(\sigma_1(r))+\gamma_{32}(\sigma_2(r))$
and
$g(r)=\gamma_{13}^{-1}(\sigma_1(r)-\gamma_{12}(\sigma_2(r)))=\gamma_{23}^{-1}(\sigma_2(r)-\gamma_{21}(\sigma_1(r)))$,
and define $M(r):=\{s_3:\, h(r)<s_3<g(r)\}.$ Let us show that
$M(r)\neq\emptyset$ for all $r>0$. If this is not true then there
exists $r^{\ast}>0$ such that $s_3^\ast:=h(r^{\ast})\ge
g(r^{\ast})$ holds.  Consider the point
$s^\ast:=(s_1^\ast,s_2^\ast,s_3^\ast):=(r^{\ast},\sigma_2(r^{\ast}),s_3^\ast).$
Then $s_{3}^{\ast}\geq
g(r^{\ast})=\gamma_{13}^{-1}(s_{1}^{*}-\gamma_{12}(s_{2}^{\ast}))$,
$s_{3}^{\ast}\geq g(r^{\ast}) =
\gamma_{23}^{-1}(s_{2}^{\ast}-\gamma_{21}(s_{1}^{\ast}))$, and
$s_{3}^{\ast}=h(r^{\ast})=\gamma_{31}(s_{1}^{\ast})+\gamma_{32}(s_{2}^{\ast})$.
In other words,
$$
\Gamma(s^\ast)=
   \begin{pmatrix}
     \gamma_{12}(s_2^{\ast})+\gamma_{13}(s_3^\ast)\\
     \gamma_{21}(s_1^{\ast})+\gamma_{23}(s_3^\ast)\\
     \gamma_{31}(s_1^{\ast})+\gamma_{32}(s_2^{\ast})
   \end{pmatrix}
\ge
\begin{pmatrix}
s_{1}^{\ast} \\
s_2^{\ast} \\
s_3^\ast
\end{pmatrix}\,,
$$
contradicting \eqref{eq:3}.  Hence $M(r)$ is not empty for all
$r>0$.

Consider the functions $h(r)$ and $g(r)$. The question is how to
choose $\sigma_3(r)\in M(r)$ such that  $\sigma_{3}\in\K_\infty$.
Note that $h(r)\in\K_\infty$.  Let $g^\ast(r):=\min_{u\geq r}
g(u)$, so that $g^{\ast}(r)\leq g(r)$ for all $r\geq0$. Since
$h(r)$ is unbounded, for all $r>0$ the set $C(r):=\arg\min_{u\geq
r} g(u)$ is compact and for all points $p\in C(r)$ the relation
$g^\ast(r)\geq g(p)>h(p)\geq h(r)$ holds. We have
$h(r)<g^\ast(r)\leq g(r)$ for all $r>0$ where $g^\ast$ is a (not
necessarily strictly) increasing function. Now take
$\sigma_{3}(r):=\frac{1}{2}(g^{\ast}(r)+h(r))$ and observe that
$\sigma_{3}\in\Kinf$ and $h(r)<\sigma_{3}(r) < g^{\ast}(r)$ for
all $r>0$. Hence $\sigma:=(\sigma_1,\sigma_2,\sigma_3)^T$
satisfies $\Gamma_\mu(\sigma(r))<\sigma(r)$ for all $r>0$.

The case where one of $\gamma_{ij}$ is not a $\K_\infty$ function
but zero can be treated similarly.

\section{Conclusions}
\label{sec:conclusions}

In this paper we have provided a method for the construction of
ISS Lyapunov functions for interconnections of nonlinear ISS
systems. The method applies for an interconnection of an arbitrary
finite number of subsystems interconnected in an arbitrary way and
satisfying a small gain condition. The small gain condition is
imposed on the nonlinear gain operator $\Gamma_\mu$ that we have
introduced here. This operator contains the information of the
topological structure of the network and the interactions between
its subsystems. An ISS Lyapunov function for such a network is
given in terms of ISS Lyapunov functions of subsystems and some
auxiliary functions. We have shown how this construction is
related to the small gain condition and mapping properties of the
gain operator $\Gamma_\mu$ and its invariant sets. Namely the
small gain condition guarantees the existence of an unbounded
vector function with path in an invariant set $\Omega$ of the
operator $\Gamma_\mu$. This auxiliary function can be used to
rescale the ISS Lyapunov functions of the individual subsystems
and aggregate them into an ISS Lyapunov function for the entire
network. The construction technique for this vector function has
been detailed as well as the construction of the composite Lyapunov
function. The constructed Lyapunov function is only locally
Lipschitz continuous, so that methods from nonsmooth analysis had
to be used. The proposed method has been exemplified for linear
systems and neural networks.

\bibliographystyle{siam}
\bibliography{literatur}
\end{document}